\documentclass[11pt, twoside]{article}
\usepackage{mathrsfs}
\usepackage{bbm}
\usepackage{cite}
\usepackage{amsfonts}
\usepackage{amsmath,amssymb}
\usepackage{multicol}

\usepackage{latexsym}
\usepackage{graphicx}
\usepackage{dcolumn}
\usepackage{bm}
\usepackage{fancyhdr}
\usepackage{mathrsfs}
\usepackage{wasysym}
\usepackage{epsf}
\usepackage{epsfig}
\usepackage{psfrag}
\usepackage{ifpdf}
\usepackage[
  bookmarks=true,
  bookmarksopen=true,
  breaklinks=true,
  colorlinks=true,
  linkcolor=blue,anchorcolor=blue,
  citecolor=blue,filecolor=blue,
  menucolor=blue,pagecolor=blue,
  urlcolor=blue]{hyperref}
\usepackage{cleveref}

\newtheorem{theorem}{Theorem}[section]
\newtheorem{lemma}[theorem]{Lemma}

\newenvironment{proof}[1][Proof]{\noindent\textbf{#1.} }{\hfill $\Box$}

\allowdisplaybreaks

\allowdisplaybreaks \numberwithin{equation}{section}
\makeatletter
\begin{document}
	
\author{\small Qingxuan Wang{$^{a}$}, Binhua Feng{$^{b,}$}\footnote{Corresponding author: this work is supported by the NSFC Grants 11801519 and 11601435.}\, , Yuan Li{$^c$}
	\renewcommand\thefootnote{}
	\footnote{{E-mail addresses}: wangqx@zjnu.edu.cn(Q. Wang), binhuaf@163.com(B. Feng)}
	\setcounter{footnote}{0}
	\\\small ${^a}$ Department of Mathematics, Zhejiang Normal University, Jinhua, 321004, China
	\\\small ${^b}$ Department of Mathematics, Northwest Normal University, Lanzhou, 730070, China
    \\\small ${^b}$ Department of Mathematics, Lanzhou University, Lanzhou, 730000, China}

\date{}
\title{\textbf{Boson Stars with Long-range  Perturbations}}\maketitle
\vspace{-0.5cm}

\begin{abstract}
	We consider the Boson satr equation with long-range perturbation given by
	$$i\partial_t \psi=\sqrt{-\triangle+m^2}\,\psi+\beta(\frac{1}{|x|^\alpha}\ast |\psi|^2)\psi-(\frac{1}{|x|}\ast |\psi|^2)\psi\ \ \ \text{on $\mathbb{R}^3$,}$$
where $\frac{1}{|x|^\alpha} (0<\alpha<1)$ denotes the long-range potential.  In contrast to the well known fact that for $\beta=0$ no maximal ground state solitary wave exists when the partical number $N=N_c$ (Chandrasekhar limiting mass) [E.H. Lieb,  H.T. Yau,  \emph{Commun. Math. Phys.}, 112 (1987), pp: 147-174 ], we  show that for $\beta>0$ and small enough, there exists  at least one maximal ground state at $N=N_c$. Moreover, for $\beta>0$, we find that for initial value $\|\psi_0\|^2_2=N_c$, the solution $\psi(t)$ is global well-posedness, and  we obtain an  ``orbital stability" of those maximal ground state solitary waves in some sense, which implies that such long-range perturbation pushes the Boson star system more stable. Finally, we analyse blow-up behaviours of maximal ground states when $\beta\rightarrow 0^+$.

\textbf{Keywords}:  Boson stars; Maximal ground state solitary wave; Well-posedness; Orbital stability; Blow-up analysis.

\end{abstract}

\section{Introduction and Main Results}
\noindent

 It is well known that under the mean field limit \cite{LiebThirring1984,Lieb-Yau1987,Elgart-Schlein2007,Michelangeli-Schlein2012}, the dynamics of  boson stars can be described by following  equation
\begin{align}\label{pseudo-relativistci-Hartree}	
i\partial_t \psi=\sqrt{-\triangle+m^2}\,\psi-(\frac{1}{|x|}\ast |\psi|^2)\psi\ \ \ \text{on $\mathbb{R}^3$,}
\end{align}
 where  $\sqrt{-\triangle+m^2}$ denotes the kinetic energy operator of a relativistic particle with mass $m>0$,  the convolution kernel $\frac{1}{|x|}$ represents the Newtonian gravitational potential in appropriate physical units, and the speed of light and Plancks constant are equal to unity.

Let us recall some important work  devoted to equation (\ref{pseudo-relativistci-Hartree}).
Lenzmann\cite{Lenzmann2007} proved local and global well-posedness for equation \eqref{pseudo-relativistci-Hartree} with initial date $\psi(0,x)=\psi_0(x)$ in $H^{s}(\mathbb{R}^3), s\geq \frac{1}{2}.$
More precisely, the following criterion implies global well-posedness
\begin{align}\label{C-1}
\int_{\mathbb{R}^3}|\psi_0(x)|^2\,dx<N_c:=\int_{\mathbb{R}^3}|Q(x)|^2\,dx,
\end{align}
where  $N_c$ is regraded as ``Chandrasekhar limiting mass"  \cite{Lieb-Yau1987}, $Q$ is a positive solution for nonlinear equation
\begin{align}\label{optimal-equ}
\sqrt{-\triangle}\,Q - (\frac{1}{|x|}\ast|Q|^2)Q = -Q,
\end{align}
which gives rise to solitary wave solutions, $\psi(t,x)=e^{it}Q(x)$, for \eqref{pseudo-relativistci-Hartree} with $m=0$. In fact, it can be shown that criterion \eqref{C-1} guaranteeing global-in-time solutions is optimal in the sense that if the initial date $\|\psi_0\|^2_2>\|Q\|^2_2=N_c$, there exist solutions which blow up in finite time; see \cite{F-Lenzmann2007}. Physically, this blow-up phenomenon indicates ``gravitational collapse" of a boson star whose mass exceeds a critical value. On the other hand, solitary waves given by $\psi(t,x)=e^{it\mu}\varphi(x)$ were considered in \cite{F-J-Lenzmann2007,Lenzmann2009A-PDE,Ze-No-Co2013-Ground,Cingolani-Secchi2015ground,Cing-Secchi2015-semicla}.
The other corresponding problems we refer to
 \cite{Herr-Lenzmann2014,Shi-Peng-2019,Cho-Ozawa2006,Fabio-Pusateri2014, F-Lenzmann2007,Lenzmann-Lewin2011}  and the references therein.

In this paper, we consider an additional $L^2$-subcritical perturbation and study
the following equation
\begin{align}\label{equation-timede}
i\partial_t \psi=\sqrt{-\triangle+m^2}\,\psi+\beta(\frac{1}{|x|^\alpha}\ast |\psi|^2)\psi-(\frac{1}{|x|}\ast |\psi|^2)\psi\ \ \ \text{on $\mathbb{R}^3$,}
\end{align}
where $\frac{1}{|x|^\alpha}$ denotes \textit{long-range potential} for $0<\alpha<1$, thus we call this perturbation as \textit{long-range} perturbation.

 Consider solitary waves of the form
\begin{align}\label{sol-wave}
\psi(t,x)=e^{it\mu}\varphi(x)
\end{align}
with some $\mu\in \mathbb{R}$. Putting (\ref{sol-wave}) into (\ref{equation-timede}) leads to the following equation
\begin{align}\label{Eular-Lagrange}
\sqrt{-\triangle+m^2}\,\varphi+\big((\frac{\beta}{|x|^\alpha}-\frac{1}{|x|})\ast |\varphi|^2\big)\varphi =\mu\, \varphi,
\end{align}
which is also an Euler-Lagrange equation for the following minimization problem
\begin{align}\label{minimizing-problem}
E(\beta,N):= \inf\{\mathcal{E}_\beta(\varphi): \varphi\in H^{1/2}(\mathbb{R}^3), \int_{\mathbb{R}^{3}}|\varphi(x)|^2\,dx = N\},
\end{align}
where
\begin{align}
\mathcal{E}_\beta(\varphi)&:= \frac{1}{2}\langle\varphi,\sqrt{-\triangle+m^2}\,\varphi\rangle+ I(\varphi),\\
\text{with}\ \ \ I(\varphi)&:=
\frac{1}{4}\int_{\mathbb{R}^3}\int_{\mathbb{R}^3}(\frac{\beta}{|x-y|^\alpha}-\frac{1}{|x-y|})|\varphi(x)|^2|\varphi(y)|^2\,dxdy,
\end{align}
and $N$ denotes the mass of system, or may denote the number of particles, $I(\varphi)$ denotes interaction energy. We refer to such minimizers $\varphi\in H^{1/2}(\mathbb{R})$ as \textbf{ground states} throughout this paper. Our goal is to study the existence, the ``orbital stability" of ground state solitary waves  and the blow-up behaviours of ground states.

We mention that such combined interactions $I(\varphi)$ can also be found in \cite{On-mini-inter-2015}, in which the authors use the concentration compactness principle to study the existence of minimizers for a variational problem with the energy of the form
$$E[\rho]:=\int_{\mathbb{R}^N}\int_{\mathbb{R}^N}K(x-y)\rho(x)\rho(y)\,dx\,dy,\ \ K(x):=\frac{1}{q}|x|^q-\frac{1}{p}|x|^p,\ \ \text{for $-N<p<q$}.$$
This energy appears in the study of many phenomena, including biological swarm, granular media, molecular dynamics simulations of matter, one can see the Introduction of \cite{On-mini-inter-2015}.

On the other hand, notice that  $I(\varphi)$ describes combined nonlocal interactions. For combined local interactions, like  power-type nonlinearities   we refer to \cite{Tao-Power-type2007,Song-stab2010, Feng-2018,Feng-2018-frac}; other problems where the energy functionals contain the combined nonlocal and local nonlinearities, one can see \cite{Bell-sici11zamp,Jeanjeanp2013,Jeanjean-zamp,zeng-zhang-sch-poisson2017} and the reference therein.

Before stating our results, we  recall from \cite{ F-J-Lenzmann2007} and \cite{Lenzmann-Lewin2011} the following Gagliardo-Nirenberg type inequality
\begin{align}\label{maininequality}
\int_{\mathbb{R}^3}(\frac{1}{|x|}\ast|\psi|^2)|\psi|^2\,dx \leq \frac{2}{N_c}\langle\psi,\sqrt{-\triangle}\psi\rangle\,\langle\psi,\psi\rangle,
\end{align}
where $\langle\cdot,\cdot\rangle$ denotes $L^2$ product, $\frac{2}{N_c}$ is the best constant given by
\begin{align}\label{best-constant}
\frac{N_c}{2}=\inf_{\psi\in H^{1/2}(\mathbb{R}^3),\psi\not\equiv 0}\frac{\langle\psi,\sqrt{-\triangle}\,\psi\rangle\,\langle\psi,\psi\rangle}{\int_{\mathbb{R}^3}(\frac{1}{|x|}\ast|\psi|^2)|\psi|^2\,dx}.
\end{align}
 Moreover, any optimizer $Q(x)$ of above inequality satisfies equation \eqref{optimal-equ} above, such that the following  identity holds (see Appendix A of \cite{Lenzmann-Lewin2011})
\begin{align}\label{Pohoza-identity}
\langle Q, \sqrt{-\triangle} Q\rangle=\frac{1}{2}\int_{\mathbb{R}^3}(\frac{1}{|x|}\ast|Q|^2)|Q|^2\,dx
=\int_{\mathbb{R}^3}|Q|^2\,dx=N_c,
\end{align}
where $N_c$ is  defined in \eqref{best-constant}.

First, we study the existence and nonexistence of ground states for problem \eqref{minimizing-problem} when $0<N<N_c$ and $N>N_c$.
\begin{theorem}\label{theorem0}
	Suppose that $m>0$, $\beta\in \mathbb{R}$ and $0<\alpha<1$. Let $E(\beta,N)$ be given in \eqref{minimizing-problem}.\vspace{0.2cm}\\
(i) If $0<N<N_c$, $\beta\leq 0$, then there exists at least one minimizer for  $E(\beta, N)$. \vspace{0.2cm}\\
(ii) If $0<N<N_c$ and $\beta>0$, then there exists at least one minimizer for  $E(\beta, N)$ when $\beta$ small enough and $N$ closes to $N_c$ enough. \vspace{0.2cm}\\
(iii) If $N>N_c$, then for any $\beta\in \mathbb{R}$, there is no minimizer for $E(\beta,N)$ such that
$E(\beta,N)=-\infty.$\vspace{0.2cm}
\end{theorem}

The case $N=N_c$ is special, we have
\begin{theorem}\label{theorem1}
	Assume that $N=N_c$, $m>0$, $\beta\in \mathbb{R}$ and $0<\alpha<1$. We have \\
	(i) If $N=N_c$ and $\beta\leq 0$, there is no minimizer for $E(\beta,N_c)$ such that
	$E(0,N_c)=0$, and $E(\beta,N_c)=-\infty$ for all $\beta<0$.\vspace{0.2cm}\\
	(ii) If $N=N_c$ and $\beta> 0$,  there exists at least one minimizer  for $E(\beta,N_c)$ with $\beta$ small enough.
\end{theorem}
\textbf{Remark A:}\\
\textbf{i)} Since there is still no minimizer for $E(\beta,N_c)$ with $\beta>0$ and $N>N_c$ (see Theroem \ref{theorem0} (iii)), We may call the minimizers in case (ii) of  this theorem as \textbf{maximal ground states}. Then the corresponding  solitary waves given in \eqref{sol-wave} can be called as \textbf{maximal ground state solitary waves.}\\
\textbf{ii)} We have no idea when $N=N_c$ and $\beta> 0$ whether there exist minimizers  for $\beta$ big enough or not.\\
\textbf{iii)}  We should mention that  the condition  $m>0$ plays an important role in the case (ii), since when $m=0$ with $\beta>0$,  $E(\beta,N_c)$ has no minimizer, one can see Theorem \ref{lemma-m-0} in the Appendix below.\vspace{0.2cm}

The existence of minizers for  $E(\beta,N)$ is based on the concentration compactness lemma obtained in \cite{F-J-Lenzmann2007}. Compare to the case (i) and (ii) of Theorem \ref{theorem0}, the case (ii) of theorem \eqref{theorem1} is harder to come by, since when $N=N_c$, for any $\varphi\in H^{1/2}(\mathbb{R}^3)$ with $\|\varphi\|^2_2=N_c$, the $H^{1/2}$ norm $\|\varphi\|_{H^{1/2}}$ can not be controlled derectly by the energy $\mathcal{E}_\beta(\varphi)$ due to the inequality \eqref{maininequality}. To overcome this, we need to prove by contradiction for case (ii) in Theorem \ref{theorem1}.

Next we will investigate an ``orbital stability" for maximal ground state solitary waves. Before doing it, we need to get the global well-posedness of Cauchy problem for eqution \eqref{equation-timede}. We have

\begin{theorem}\label{Global-well-theorem}
	Suppose that $m>0,\beta\in \mathbb{R}$ and $0<\alpha<1$. Then the unique solution of \eqref{equation-timede}  is global in time (i.e., $T=\infty$), provided that one of the following conditions for initial value $\psi_0$ holds:\\
	(i) $0<\|\psi_0\|^2_2<N_c$;\\
	(ii) $\|\psi_0\|^2_2=N_c$ such that $\beta>0$.
\end{theorem}
\textbf{Remark B:} \\
\textbf{i)} To our knowledge, there is no other result consdering the well-posedness of $\psi(t)$ when $\|\psi_0\|^2_2=N_c.$ Even the non-perturbation case,  where $\beta=0$ and $m>0$, has not been settled, one may hope that there exists blow-up solution.\\
\textbf{ii)} We may conject that, if $\beta>0$, for any $\epsilon>0$, there exists  initial value $\psi_0\in H^{1/2}(\mathbb{R}^3)$ with $\|\psi_0\|^2_2=N_c+\epsilon$, such that the solution $\psi(t)$  blow up in finite time. In this case, the case (ii) in above theorem gives an example of nonexistence of minimal blow-up solution like \cite{Merle1996}.\vspace{0.2cm}

Now we address ``orbital stability" of maximal ground state solitary waves
\begin{align}\label{soli-wave}
\psi(t,x)=e^{it\mu}\varphi(x),
\end{align}
where $\varphi \in H^{1/2}(\mathbb{R}^3)$ satisfying $\|\varphi\|^2_2=N_c$, is a ground state of $E(\beta,N_c)$.
\begin{theorem}\label{theorem-stable}
	Suppose that $m>0,\beta>0$ and $N=N_c$. Let
	\begin{align}
	S_{N_c}:=\{\varphi\in H^{1/2}(\mathbb{R}^3):\mathcal{E}_\beta(\varphi)=E(\beta,N_c),\|\varphi\|^2_2=N_c\}.
	\end{align}
	Then the solitary waves given in \eqref{sol-wave} with  $\varphi\in S_{N_c}$ are stable for  $\beta$ small enough in the following sense. Let $\psi(t)$ denotes the solution of \eqref{pseudo-relativistci-Hartree} with initial condition $\psi_0\in H^{1/2}(\mathbb{R}^3)$. For every $\epsilon >0 $, there exists $\delta>0$ such that if
	\begin{align*}
	\inf_{\varphi \in S_{N_c}}\|\psi_0-\varphi\|_{H^{1/2}}\leq \delta \ \ \ \ \text{such that} \ \ \ \|\psi_0\|^2_2\leq N_c,
	\end{align*}
	then
	\begin{align*}
	\sup_{t\geq 0}\inf_{\varphi\in S_{N_c}}\|\psi(t)-\varphi\|_{H^{1/2}}\leq \epsilon.
	\end{align*}

\end{theorem}
\textbf{Remark C:}  The condition $\|\psi_0\|^2_2\leq N_c$ is necessary to guarantee the solution $\psi(t)$ is global, one can see Theorem \ref{Global-well-theorem}.

\begin{theorem}\label{theorm-asymp-beta}
	Under the assumptions of Therorem \ref{theorem1}, let $\varphi_{\beta}(x)$ be a minimizer of $E(\beta,N_c)$. Given a sequence $\{\beta_k\}$ with $\beta_k \rightarrow 0$ as $k\rightarrow \infty$, there exists a subsequence (still denoted by $\{\beta_k\}$) such that\\
	(i)\begin{align}\label{concentrate-limit}
	\lim_{k\rightarrow\infty}\beta^{\frac{3}{2(1+\alpha)}}_k \varphi_{\beta_k} (\beta^{\frac{1}{1+\alpha}}_k x)= \gamma^{3/2} Q(\gamma |x-y_0|)
	\end{align}
	strongly in $L^q(\mathbb{R}^3)$ for some $2\leq q< 3$, $y_0\in \mathbb{R}^3$ and
	$$\gamma=\bigg(\frac{\langle Q,\frac{m^2}{\sqrt{-\triangle} }Q\rangle}{\int_{\mathbb{R}^3}(\frac{1}{|x|^\alpha}\ast|Q|^2)|Q|^2\,dx}\bigg)^{\frac{1}{1+\alpha}}.$$
	(ii)
	\begin{align}\label{energy-limit}
	\lim_{\beta_k\rightarrow 0} \frac{E(\beta_k,N_c)}{\beta^{\frac{1}{1+\alpha}}_k}=\frac{1}{2}\langle Q,\frac{m^2}{\sqrt{-\triangle} }Q\rangle^{\frac{\alpha}{1+\alpha}}\bigg(\int_{\mathbb{R}^3}(\frac{1}{|x|^\alpha}\ast|Q|^2)|Q|^2\,dx\bigg)^\frac{1}{1+\alpha}
	\end{align}
\end{theorem}
\textbf{Remark D:}\\
\textbf{i)}: \eqref{concentrate-limit} also means that
$$|\varphi_{\beta_k}(x)|^2\rightharpoonup N_c\delta(y-y_0).$$
in the distribution sense, where $\delta(x)$ denotes Dirac delta function.\\
\textbf{ii)}: We have mentioned above that the ground state $\varphi_{\beta_k}$ satisfies \eqref{Eular-Lagrange}, then $\psi_{\beta_k}(t)=e^{it\mu_k}\varphi_{\beta_k}(x)$ satisfies \eqref{equation-timede} uniquely  with initial date $\psi_{\beta_k}(0)=\varphi_{\beta_k}$ (by Theorem \eqref{Global-well-theorem} (ii) ), this means that $|\psi_{\beta_k}(t)|^2\rightharpoonup N_c\delta(y-y_0)$ in the distribution sense for all $t\geq 0$. On the other hand, we have $\mu_k\sim -{\beta_k}^{-\frac{1}{1+\alpha}}$ as $\beta_k\rightarrow 0^+$ (just by \eqref{mu-beta}, \eqref{optimal-energy-estimates} and \eqref{psi-beta-hartree-est}). \vspace{0.2cm}

This theorem shows that, when the action of long-range potential is small, then the mass of  ground state at $N=N_c$ will concentrate.
Such similar blow-up results  appeared in studying Bose-Einstein condensations with attractive interaction described by Gross-Pitaevskii functional, one can see \cite{Guo_Seiringer2014,Guo_zheng_zhou2015,Deng_Guo_Lu2015,Wang-Zhao2017}. There are some blow-up results for Boson stars,  Guo and  Zeng\cite{Guo-Zeng-Ground-Pseud2017} studied the asymptotic behaviour as $N\nearrow N_c$ for different self-interacting potentials, Nguyen\cite{Ngugen2017-On-Blow} and Yang\cite{Yang-Yang} studied it for different external potentials. In this paper we  forcus on the asymptotic  behaviour of ground states at $N=N_c$ when $\beta\rightarrow 0^+$. This blow-up analysis is more difficult in contrast to the works of \cite{Guo-Zeng-Ground-Pseud2017,Ngugen2017-On-Blow,Yang-Yang},
we need some technical arguments,  due to the lack of compactness and $H^{1/2}$ norm of ground state can not controlled derectly by energy $E(\beta,N_c)$. The first key point to this theorem is to obtain an optimal estimate for $E(\beta,N_c)$, to do this we need to employ the concentration-compactness arguments to obtain the lower bound of $E(\beta,N_c)$. The second key point is to obtain an optimal estimate of $\langle\varphi_{\beta},\sqrt{-\triangle}\, \varphi_\beta\rangle$, the upper bound is harder to come by, we should employ  scaling arguments and prove by contradiction, this is quite different from the mentioned papers.

This paper is organized as follows: in Section 2, we study the existence and nonexistence of minimizers for $E(\beta,N)$.  In Section 3, we consider the Cauchy problem of equation \eqref{sol-wave} and study the well-posedness and orbital stability; in Section 4, we prove asymptotic behaviour of ground states for $E(\beta,N_c)$ as $\beta\rightarrow 0^+$.

\textbf{Notation:}\\
- $\langle\cdot,\cdot\rangle$ denotes $L^2$ product.\\
- $\|\cdot\|_p$ denotes the $L^p(\mathbb{R}^3)$ norm for $p\geq 1$.\\
- $\rightharpoonup$ denotes weakly converge, $\ast$ stands for convolution on $\mathbb{R}^3$.\\
- $\beta\rightarrow  0^+$ denotes $\beta\rightarrow  0^+$ with $\beta>0$.\\
- $a\lesssim b$ denotes $a\leq Cb$ for some appropriate constant $C>0$.\\
- For symbol $\sqrt{-\triangle +m^2}$ and $H^{1/2}(\mathbb{R}^3)$, one can see \cite{F-J-Lenzmann2007}.

\section {Existence and nonexistence of minimizers for $E(\beta,N)$}

\subsection{Some properties of the energy $E(\beta,N)$}
\begin{lemma}\label{less-zero} Suppose that $m>0$ and $\beta\in \mathbb{R}$.
	
	(I) When $0<N<N_c$, we have\\
	(i) If  $\beta\leq 0$,  $E(\beta, N)<\frac{1}{2}mN$ for  $0<N<N_c$ such that $E(\beta,N)$ is strictly decreasing in $N$. \\
	(ii) If $\beta>0$,   $E(\beta, N)<\frac{1}{2}mN$ for $\beta$ small enough and $N$ closes to $N_c$ enough.
	
(II) When $N=N_c$, we have \\
(iii) If $N=N_c$ and $\beta\leq 0$,
$E(0,N_c)=0$, and $E(\beta,N_c)=-\infty$ for all $\beta<0$;\\
(iv) If $\beta>0$, then
$E(\beta,N_c)<\frac{1}{2}mN_c$ for  $\beta$ small enough.

(III)When $N>N_c$, then for any $\beta\in \mathbb{R}$,
$E(\beta,N)=-\infty.$
\end{lemma}
\begin{proof} Note that, by (2.24) of \cite{F-J-Lenzmann2007} we know that $E(0,N)-\frac{1}{2}mN<0$ for $0<N<N_c$, in fact, $E(0,N)-\frac{1}{2}mN$ is equal to $E_0(N)$ of \cite{F-J-Lenzmann2007} by taking $v=0$. On the other hand, it is easy to see that $E(\beta, N)\leq E(0,N)$ for $\beta<0$ and $0<N<N_c$. Following the same arguments as Lemma 2.3 of \cite{F-J-Lenzmann2007}, one can easy to check that $E(\beta,N)$ is strictly decreasing. Thus we obtain the case (i) of this lemma.
	
	To prove the case (ii), now let  $Q^{\lambda}= \lambda^{3/2}Q(\lambda x)$ with $\lambda>0$, where $Q$ is an optimizer of  (\ref{maininequality}).  One can check that $Q_\lambda$ also satisfies \eqref{Pohoza-identity}. We have
\begin{align}
&E(\beta,N)\leq\mathcal{E}_\beta(\sqrt{\frac{N}{N_c}}\,\cdot Q^{\lambda})\notag\\
&=\frac{N}{N_c}\bigg\{\frac{1}{2}\langle Q^{\lambda}, \sqrt{-\triangle+m^2}\,Q^{\lambda}\rangle
-\frac{N}{N_c}\,\cdot \frac{1}{4}\int_{\mathbb{R}^3}\big(\frac{1}{|x|}\ast |Q^{\lambda}|^2\big)|Q^{\lambda}|^2\,dx\notag\\ &\ \ \ \ \ +\frac{N}{N_c}\,\cdot\frac{\beta}{4}\int_{\mathbb{R}^3}(\frac{1}{|x|^\alpha}\ast|Q^\lambda|^2)|Q^\lambda|^2\,dx\bigg\}\notag\\ 
&=\frac{N}{N_c}\bigg\{\frac{1}{2}\langle Q^{\lambda}, (\sqrt{-\triangle+m^2}-\sqrt{\triangle})Q^{\lambda}\rangle
+(1-\frac{N}{N_c})\,\cdot \frac{1}{4}\int_{\mathbb{R}^3}\big(\frac{1}{|x|}\ast |Q^{\lambda}|^2\big)|Q^{\lambda}|^2\,dx\notag\\ &\ \ \ \ +\frac{N}{N_c}\,\cdot\frac{\beta}{4}\int_{\mathbb{R}^3}(\frac{1}{|x|^\alpha}\ast|Q^\lambda|^2)|Q^\lambda|^2\,dx\bigg\}\ \ \ (\text{obtained by \eqref{Pohoza-identity}})\notag\\
&\leq \frac{N}{N_c}\bigg\{ \frac{m^2}{\lambda}\int_{\mathbb{R}^3}\frac{1}{4\sqrt{-\triangle}}|Q(x)|^2\,dx+\frac{\lambda(N_c-N)}{4 N_c}\int_{\mathbb{R}^3}\big(\frac{1}{|x|}\ast |Q|^2\big)|Q|^2\,dx \notag\\ &\ \ \ \ +\frac{N}{N_c}\,\cdot\frac{\beta\lambda^\alpha}{4}\int_{\mathbb{R}^3}(\frac{1}{|x|^\alpha}\ast|Q|^2)|Q|^2\,dx\bigg\}. \ \ \ \text{(since $\sqrt{-\triangle+m^2}-\sqrt{-\triangle}\leq \frac{1}{2\sqrt{-\triangle}}$)}\label{e-N-beta}
\end{align}
Since $N_c-N>0$ and $\beta>0$, now we take
$$\lambda=\frac{1}{(N_c-N)^{\frac{1}{2}}+\beta^{\frac{1}{1+\alpha}}},$$
then
\begin{align*}
(N_c-N)\lambda&=\frac{N_c-N}{(N_c-N)^{\frac{1}{2}}+\beta^{\frac{1}{1+\alpha}}}\leq (N_c -N)^{\frac{1}{2}};\\
 \beta\lambda^{\alpha}&=\frac{\beta}{[(N_c-N)^{\frac{1}{2}}+\beta^{\frac{1}{1+\alpha}}]^{\alpha}}\leq \beta^{\frac{1}{1+\alpha}}.
\end{align*}
It follows that there exist $C_1>0$ and $C_2>0$ independent of $N$ and $\beta$ such that
\begin{align}
E(\beta,N)\leq C_1(N_c -N)^{1/2}+C_2\beta^{\frac{1}{1+\alpha}}.
\end{align}
Then, when $N$ close $N_c$ and $\beta$ small enough, we have $E(\beta,N)<\frac{1}{2}mN$. This completes the proof of case (ii).

To prove the case (iii) and (iv), the same as \eqref{e-N-beta} and let $N=N_c$ in \eqref{e-N-beta}, we have
\begin{align}\label{up-bound-estimate1}
E(\beta,N_c)\leq
 \frac{m^2}{\lambda}\int_{\mathbb{R}^3}\frac{1}{2\sqrt{-\triangle}}|Q(x)|^2\,dx+\frac{\beta \lambda^\alpha}{4}\int_{\mathbb{R}^3}(\frac{1}{|x|^\alpha}\ast|Q|^2)|Q|^2\,dx.
\end{align}
Since $\beta\leq 0$ in case (iii), just let $\lambda\rightarrow \infty$ in \eqref{up-bound-estimate1}, we can obtain case (iii).

To prove case (iv), since $\beta>0$, take the infimum over $\lambda$ in \eqref{up-bound-estimate1}, then
\begin{align}\label{upper-bound}
E(\beta,N_c)\leq  C \beta^{\frac{1}{1+\alpha}}.
\end{align}
It follows that for $\beta$ small enough, $C \beta^{\frac{1}{1+\alpha}}<\frac{1}{2}mN_c$, i.e., $E(\lambda, N_c)<\frac{1}{2}mN_c$. Thus we complete the proof of case (iv).

To prove the case (III),  the same as \eqref{e-N-beta} and let $N>N_c$ in \eqref{e-N-beta}. Note that $\frac{\lambda(N_c-N)}{4N_c}<0$ and $0<\alpha<1$, let $\lambda \rightarrow \infty$, then $\eqref{e-N-beta}\rightarrow -\infty$, thus $E(\beta,N)=-\infty$.
\end{proof}

\begin{lemma}\label{binding-ineq} The following strictly binding inequality
\begin{align}\label{subadd-ineq}
E(\beta,N)< E(\beta, \lambda)+E(\beta, N-\lambda)
\end{align}
holds for any $0<\lambda<N$,  when $N$ and $\beta$  satisfy one of the three conditions.\\
(i) $0<N<N_c$ and $\beta\leq 0$;\\
(ii) $0<N<N_c$ and $\beta>0$,  $\beta$ small enough such that $N$ closes to $N_c$ enough;\\
(iii) $N=N_c$, $\beta>0$ such that small enough.
\end{lemma}
\textbf{Remark:} The condition (i), (ii), (iii) here correspond to the assumptions of (i), (ii), (iv) in Lemma \ref{less-zero}, which guarantee the energy $E(\beta, N)<\frac{1}{2}mN$. We mention that in \cite[Lemma 2.3]{F-J-Lenzmann2007}, the authors showed that such binding inequality holds for $0<N<N_c$, in what follows, we prove that such inequlity also holds at the threshold $N=N_c$. \vspace{0.2cm}\\
\begin{proof} In what follows we only prove for condition (iii), for (i) and (ii) one can just follow the same arguments.
	
	Now Let $\widetilde{E}(\beta,N):=E(\beta,N)-\frac{1}{2}mN$, then \eqref{subadd-ineq} is equivalent to
	\begin{align}\label{subadd-ineq1}
	\widetilde{E}(\beta,N)< \widetilde{E}(\beta, \lambda)+\widetilde{E}(\beta, N-\lambda)
	\end{align}
	for any $0<\lambda<N$. Next we will prove \eqref{subadd-ineq1} for case (iii).
	
Note that, for any $\epsilon >0$, there exists $v\in H^{1/2}(\mathbb{R}^3)$, $\|v\|^2_2=N$ such that $\mathcal{E}_\beta(v)\leq  E(\beta, N)+\epsilon$. Hence for any $\theta>1$
\begin{align}
&\widetilde{E}(\beta, \theta N)=E(\beta,\theta N)-\frac{\theta}{2}mN\leq  \mathcal{E}_\beta(\sqrt{\theta}v)-\frac{\theta}{2}mN\notag\\
&=(\theta-\theta^2)\frac{1}{2}\langle v,(\sqrt{-\triangle+m^2}-m)\,v\rangle+\theta^2(\mathcal{E}_\beta(v)-\frac{1}{2}mN)\notag\\\label{E_scal-ineq}
&\leq  \theta^2(\mathcal{E}_\beta(v)-\frac{1}{2}mN)\leq \theta^2(\widetilde{E}(\beta, N)+\epsilon).
\end{align}
First we claim that for any $0<\lambda<N_c$
\begin{align}\label{beta-lambda}
\widetilde{E}(\beta,N_c)<\frac{N_c}{\lambda}\widetilde{E}(\beta, \lambda).
\end{align}
Indeed, if $\widetilde{E}(\beta, \lambda)\geq 0$,  then (\ref{beta-lambda}) holds obviously since $\widetilde{E}(\beta,N_c)<0$ for  $\beta$ small enough (by Lemma \ref{less-zero}(iv)).
If $\widetilde{E}(\beta, \lambda)< 0$,
choose  $\theta=\frac{N_c}{\lambda}$ and $N=\lambda $ in  inequality (\ref{E_scal-ineq}), and let $\epsilon <(\theta^{-1}-1)\widetilde{E}(\beta,N)$, it follows that (\ref{beta-lambda}) holds.
Thus, the claim holds.

In the same way we have
\begin{align}\label{beta-lambda1}
\widetilde{E}(\beta,N_c)<\frac{N_c}{N_c-\lambda}\widetilde{E}(\beta, N_c-\lambda).
\end{align}
Combing \eqref{beta-lambda} with \eqref{beta-lambda1}, then \eqref{subadd-ineq1} holds for condition (iii).
\end{proof}

\subsection{The proof of Theorem \ref{theorem0} and \ref{theorem1}}

First, we claim that
\begin{lemma} \label{minimi-seq-bound}
(1) Under the assumptions of case (i) and (ii) in Theorem \ref{theorem0}, any minimizing sequence for $E(\beta,N)$ is uniformly bounded in $H^{1/2}(\mathbb{R}^2)$.\\
(2) Under the assumptions of case (ii) in Theorem \ref{theorem1}, i.e. $N=N_c$ and $\beta>0$. Then any minimizing sequence for $E(\beta,N_c)$ is uniformly bounded in $H^{1/2}(\mathbb{R}^2)$.
\end{lemma}

\begin{proof}
	Notice that  by Hardy-Littlewood-Sobolev inequality, interpolation inequality and Sobolev inequality,  we have
	\begin{align}\label{alpha-up-bound1}
	 \bigg|\int_{\mathbb{R}^3}(\frac{1}{|x|^\alpha}\ast|\psi|^2)|\psi|^2\,dx\leq C\|\psi\|^4_{\frac{12}{6-\alpha}}\leq C_1 \|\psi\|^{2\alpha}_{3} \|\psi\|^{2(2-\alpha)}_{2}\leq C_2\langle\psi,\sqrt{-\triangle}\,\psi\rangle^{\alpha}.
	\end{align}
 For case (ii) in Theorem \ref{theorem0}, since $\beta>0$, use the fact $\sqrt{-\triangle+m^2}\geq \sqrt{-\triangle}$ by  \eqref{maininequality}
 \begin{align}\label{lower1}
 \mathcal{E}_\beta(\psi)\geq \frac{1}{2}(1-\frac{N}{N_c})\langle\psi,\sqrt{-\triangle}\,\psi\rangle-\frac{1}{2}mN.
 \end{align}
 For case (i) in Theorem \ref{theorem0}, since $\beta\leq 0$, by \eqref{maininequality} and \eqref{alpha-up-bound1}
\begin{align}\label{lower2}
\mathcal{E}_\beta(\psi)\geq \frac{1}{2}(1-\frac{N}{N_c})\langle\psi,\sqrt{-\triangle}\,\psi\rangle+\beta\,C_2\langle\psi,\sqrt{-\triangle}\,\psi\rangle^{\alpha}.
\end{align}
For any minimizing sequence $\{\psi_n\}$, since $0<N<N_c$ and $0<\alpha<1$, then $1-\frac{N}{N_c}>0$, and $\sup_n \langle\psi_n,\sqrt{-\triangle}\,\psi_n\rangle\leq  C<\infty$ thanks to \eqref{lower1} and \eqref{lower2}. Thus we obtian the case (1).

Next we prove the case (2). Let $\{\psi_n(x)\}$ be  a minimizing sequence  of $E(\beta, N_c)$, such that
\begin{align}\label{minimi-seque-ineq}
E(\beta, N_c)\leq \mathcal{E}_{\beta}(\psi_n)\leq E(\beta, N_c)+\frac{1}{n}.
\end{align}
On the contrary, we now suppose that $\{\psi_n(x)\}$ is unbounded in $H^{1/2}(\mathbb{R}^3)$. Then there exists a subsequence $\{\psi_n\}$ (still denoted by $\{\psi_n\}$), such that
\begin{align}\label{gradi-est-inft}
\langle \psi_n, \sqrt{-\triangle}\psi_n\rangle \rightarrow +\infty. \ \ (n\rightarrow \infty)
\end{align}
Since $\sqrt{-\triangle+m^2}\geq \sqrt{-\triangle}$, by \eqref{maininequality} and \eqref{minimi-seque-ineq}
we have
\begin{align}\label{u-n-alpha} \frac{\beta}{4}\int_{\mathbb{R}^3}(\frac{1}{|x|^\alpha}\ast|\psi_n|^2)|\psi_n|^2\,dx\leq \mathcal{E}_{\beta}(\psi_n)\leq E(\beta,N_c)+\frac{1}{n},
\end{align}
and also
\begin{align}\label{gradi-u-Hartree-est}
0\leq \frac{1}{2}\langle \psi_n, \sqrt{-\triangle}\,\psi_n\rangle-\frac{1}{4}\int_{\mathbb{R}^{3}}(\frac{1}{|x|}\ast|\psi_n|^2)|\psi_n|^2\,dx\leq E(\beta,N_c)+\frac{1}{n}.
\end{align}

Define now
\begin{align}\label{epsi-un}
\epsilon^{-1}_n:= \langle \psi_n, \sqrt{-\triangle}\psi_n\rangle.
\end{align}
Then $\epsilon_n\rightarrow 0$ as $n\rightarrow \infty$. Define
\begin{align}\label{w-n}
\tilde{w}_n(x):=\epsilon^{\frac{3}{2}}_n \psi_n(\epsilon_n x).
\end{align}
By (\ref{epsi-un}) and (\ref{gradi-u-Hartree-est}), we have
\begin{align}
0\leq \frac{1}{2}\langle \tilde{w}_n, \sqrt{-\triangle}\,\tilde{w}_n\rangle-\frac{1}{4}\int_{\mathbb{R}^{3}}(\frac{1}{|x|}\ast|\tilde{w}_n|^2)|\tilde{w}_n|^2\,dx\leq \epsilon_n (E(\beta,N_c)+\frac{1}{n})\rightarrow 0.
\end{align}
Therefore, we can conclude that
\begin{align}\label{tilde-w-bound1}
\langle \tilde{w}_n, \sqrt{-\triangle}\,\tilde{w}_n\rangle=1,\ \ \
M\leq \int_{\mathbb{R}^{3}}(\frac{1}{|x|}\ast|\tilde{w}_n|^2)|\tilde{w}_n|^2\,dx \leq \frac{1}{M}.
\end{align}
for some constant $M>0$, independent of $n$.

We claim that,
there exist a sequence $\{y_{\epsilon_n}\}$ and positive constants $R_0$ and $\eta$ such that
\begin{align}\label{eta-n_blow}
\liminf_{\epsilon_n\rightarrow 0}\int_{B(y_{\epsilon_n}, R_0)}|\tilde{w}_n(x)|^2\,dx\geq \eta >0.
\end{align}
Otherwise, by Lemma A.1 in
\cite{F-J-Lenzmann2007}, we have
$$\int_{\mathbb{R}^{3}}(\frac{1}{|x|}\ast|\tilde{w}_n|^2)|\tilde{w}_n|^2\,dx \rightarrow 0,$$
 this contradicts (\ref{tilde-w-bound1}).

Let
\begin{align}\label{w-n-2}
w_{n}(x)=\tilde{w}_{\epsilon_n}(x+ y_{\epsilon_n})=\epsilon^{\frac{3}{2}}_n \psi_{n}(\epsilon_n x +\epsilon_n y_{\epsilon_n}).
\end{align}
Then it follows from (\ref{eta-n_blow}) that
 \begin{align}\label{w-n-below-estimate}
\liminf_{\epsilon_n\rightarrow 0}\int_{B(0, R_0)}|w_{n}(x)|^2\,dx\geq \eta >0.
\end{align}
 \textbf{Claim:} There exists $\eta_0>0$ such that
 \begin{align}\label{w-n-below-estimate1}
\int_{\mathbb{R}^3}(\frac{1}{|x|^\alpha}\ast|w_n|^2)|w_n|^2\,dx\geq \eta_0 >0.
\end{align}
Indeed, let $R_0$ be given in \eqref{w-n-below-estimate},
\begin{align*}
&\int_{\mathbb{R}^3}(\frac{1}{|x|^\alpha}\ast|w_n|^2)|w_n|^2\,dx\\
&\geq \int_{B(0,R_0)}\int_{B(0,R_0)}\frac{1}{|x-y|^\alpha}|w_n(x)|^2|w_n(y)|^2\,dxdy\\
&\geq \int_{B(0,R_0)}\int_{B(0,R_0)}\frac{1}{2R^\alpha_0}|w_n(x)|^2|w_n(y)|^2\,dxdy\\
&=\frac{1}{2R^\alpha_0}\big(\int_{B(0,R_0)}|w_n(x)|^2\,dx\big)^2\geq \frac{\eta^2}{2R^\alpha_0}\  \ \ \text{(by \eqref{w-n-below-estimate}).}
\end{align*}
Thus  the claim holds.

On the other hand, note that by (\ref{w-n-2}) and (\ref{u-n-alpha}) we have
\begin{align*}
\frac{\beta \epsilon^{-\alpha}_n}{4}\int_{\mathbb{R}^3}(\frac{1}{|x|^\alpha}\ast|w_n|^2)|w_n|^2\,dx=
\frac{\beta}{4}\int_{\mathbb{R}^3}(\frac{1}{|x|^\alpha}\ast|\psi_n|^2)|\psi_n|^2\,dx \leq  E(\beta,N_c)+\frac{1}{n},
\end{align*}
 which implies that
 \begin{align}\label{alpha-0}
\int_{\mathbb{R}^3}(\frac{1}{|x|^\alpha}\ast|w_n|^2)|w_n|^2\,dx\leq C\epsilon^{\alpha}_n \rightarrow 0,\ \ \ \epsilon_n\rightarrow 0.
\end{align}
This contradicts (\ref{w-n-below-estimate1}). Therefore, $\{\psi_n(x)\}$ is bounded uniformly in $H^{1/2}(\mathbb{R}^3)$, we complete the lemma.
\
\end{proof}\vspace{0.2cm}\\

Now we go to prove the Theorem \ref{theorem0} and Theorem \ref{theorem1}. Notice that, the case (iii) of \ref{theorem0} has been proved in case (III)  of Lemma \ref{less-zero} above, and the case (i) of Theorem \ref{theorem1} has been proved in case (iii) of Lemma \ref{less-zero} above. The rest, the case (i) and (ii) of  Theorem \ref{theorem0}, the case (ii) of Theorem \ref{theorem1} can be proved by standard concentration-compactness arguments as \cite{F-J-Lenzmann2007} by combing with Lemma \ref{less-zero}, \ref{binding-ineq} and \ref{minimi-seq-bound} above. To this reason, we only give the proof for the case (ii) of Theorem \ref{theorem1} below,  and the same to the other two cases.

\textbf{The end of the proof of Theorem \ref{theorem1} (ii):} Since $\{\psi_n\}$ is bounded uniformly in $H^{1/2}(\mathbb{R}^3)$, so now we can use the concentration-compactness lemma (Lemma \ref{concen-compact} below)
(1)\textit{Vanishing does not occur}\\
If vanishing occurs, it follows from Lemma \ref{vanishing-alpha} we have
\begin{align*}
\lim_{n\rightarrow \infty}\int_{\mathbb{R}^3}(\frac{1}{|x|^\alpha}\ast|\psi_{n}|^2)|\psi_{n}|^2\,dx=0;\ \ \lim_{n\rightarrow \infty}\int_{\mathbb{R}^3}(\frac{1}{|x|}\ast|\psi_{n}|^2)|\psi_{n}|^2\,dx=0.
\end{align*}
Since $\sqrt{-\triangle+m^2}-m\geq 0$, it follows that
\begin{align*}
E(\beta,N_c)=\lim_{n\rightarrow\infty}\mathcal{E}_\beta(\psi_{\beta_k})
= \lim_{n\rightarrow\infty} \big\{\frac{1}{2}\langle\psi_{n},(\sqrt{-\triangle+m^2}\,\psi_{n}\rangle\big\}\geq \frac{1}{2}mN_c.
\end{align*}
which  contradicts lemma \ref{less-zero}. Therefore, vanishing does not occur.\\
(2)\textit{Dichotomy does not occur:} \\
If the dichotomy occurs, by Lemma \ref{concen-compact}(iii) below, then there exists $\lambda\in (0,N_c)$ such that, for every $\epsilon>0$, there exists  two bounded dichotomy subsequences in $H^{1/2}(\mathbb{R}^3)$ denoted by $\{\psi^1_{n_k}\}$ and $\{\psi^2_{n_k}\}$ with
\begin{align}
\lambda-\epsilon\leq \|\psi^1_{n_k}\|^2_2\leq \lambda+\epsilon,\ \ (N_c-\lambda)-\epsilon\leq \|\psi^2_{n_k}\|^2_2\leq  (N_c-\lambda)+\epsilon
\end{align}
for $k$ sufficiently large. Moreover,  \eqref{binding-est1} and \eqref{binding-est2} allow us to deduce that, there exists $r_1(k)$ and $r_2(\epsilon)$ such that
\begin{align*}
\mathcal{E}_{\beta}(\psi_{n_k})-\mathcal{E}_{\beta}(\psi^1_{n_k})-\mathcal{E}_{\beta}(\psi^2_{n_k})\geq -r_1(k)-r_2(\epsilon),
\end{align*}
where $r_1(k)\rightarrow 0$ as $k\rightarrow\infty$, and
$r_2(\epsilon)\rightarrow 0$ as $\epsilon\rightarrow 0$. Passing to limits  $k\rightarrow\infty$ and $\epsilon\rightarrow 0$, and by  continuity of $E(\beta,N)$ in $N$, we deduce that
$$E(\beta,N_c)\geq E(\beta,\lambda)+E(\beta,N_c-\lambda).$$
This contradicts \eqref{subadd-ineq}.

Therefore, compactness happens. The same arguments as the proof of part i) of Theorem 2.1 in \cite{F-J-Lenzmann2007}, then up to a translation, the  sequence $\{\psi_{n_k}\}$ are relatively compact in $H^{1/2}(\mathbb{R}^3)$. Thus there exists $\{y_k\}$ with $y_k\in \mathbb{R}^3$ such that $\widetilde{\psi}_{n_k}=\psi_{n_k}(x+y_k)$ satisfies
\begin{align}
\widetilde{\psi}_{n_k}\rightarrow \varphi, \ \ \text{strongly in $H^{1/2}(\mathbb{R}^3)$.}
\end{align}
such that $\varphi$ is a minimizer of $E(\beta,N_c)$. This completes the proof of case (ii) of Theorem \ref{theorem1}.

\section{The Cauchy problem}

In this section,  we consider the following Cauchy  problem
\begin{equation}\label{e}
\left\{
\begin{array}{l}
i\partial_t \psi=\sqrt{-\triangle+m^2}\,\psi+\beta(\frac{1}{|x|^\alpha}\ast |\psi|^2)\psi-(\frac{1}{|x|}\ast |\psi|^2)\psi\ \ \ \text{on $\mathbb{R}^3$,} \\
\psi(0,x) = \psi_0 (x) ,%
\end{array}%
\right.
\end{equation}
where $\psi:[0,T)\times \mathbb{R}^3\rightarrow \mathbb{C}$ is the complex valued function, $\psi_0 \in H^{1/2}$, $0<T\leq \infty$, $0<\alpha<1, \beta\in \mathbb{R}$.

we will prove the local well posedness and global well-posedness and finally obtain an orbital stability.

\subsection{Local Well-posedness}

In this part we apply the arguments from   \cite{Lenzmann2007} with some modifications.

 We put
 \begin{align}\label{A-fu}
 A:=\sqrt{-\triangle+m^2},\ \ \ F(u):=\beta(\frac{1}{|x|^\alpha}\ast |u|^2)u-(\frac{1}{|x|}\ast |u|^2)u.
 \end{align}
 For $s\in \mathbb{R}$, let
 \begin{align}
 D^s:=(-\triangle)^\frac{s}{2}
 \end{align}

 From \cite{Lenzmann2007}, we know that to prove the local well posedness, we only need to show the nonlinearity $F(u)$ is locally Lipschitz continuous from $H^{1/2}(\mathbb{R}^3)$ into itself. Notice that,  Lemma 1 of \cite{Lenzmann2007} has shown that  $(\frac{1}{|x|}\ast |u|^2)u$ is locally Lipschitz continuous, it is sufficiently to prove for $(\frac{1}{|x|^\alpha}\ast |u|^2)u$. First we show following  key estimates.

\begin{lemma}\label{estimate-h-4}
Suppose that $0<\alpha<1$,	for any $u,v\in H^{1/2}(\mathbb{R}^3)$, we have
\begin{align}
\|\frac{1}{|x|^\alpha}\ast|u|^2\|_{\infty}\,&\lesssim \|u\|^2_{H^{1/2}}\label{H-alpha-inf}\\
\|\frac{1}{|x|^\alpha}\ast (|u|^2-|v|^2)\|_\infty&\lesssim (\|u\|_{H^{1/2}}+\|v\|_{H^{1/2}})\|u-v\|_{H^{1/2}},\label{H-alpha-inf1}
\end{align}
and
\begin{align}\label{H-alpha-6}
\|\frac{1}{|x|^\alpha}\ast (|u|^2-|v|^2)\|_{\frac{6}{\alpha}}&\lesssim (\|u\|_{H^{1/2}}+\|v\|_{H^{1/2}})\|u-v\|_{H^{1/2}}\\ \label{H-alpha-6-1}
\|D^{1/2}\big(\frac{1}{|x|^\alpha}\ast (|u|^2-|v|^2)\big)\|_6&\lesssim (\|u\|_{H^{1/2}}+\|v\|_{H^{1/2}})\|u-v\|_{H^{1/2}}
\end{align}
\end{lemma}
\begin{proof}
	\begin{align*}
	 &\|\frac{1}{|x|^\alpha}\ast|u|^2\|_{\infty}=\sup_{y\in\mathbb{R}^3}\int_{\mathbb{R}^3}\frac{|u(x)|^2}{|x-y|^\alpha}\,dx\\
	&=\sup_{y\in\mathbb{R}^3}\big\{\int_{|x-y|\geq 1}\frac{|u(x)|^2}{|x-y|^\alpha}\,dx+\int_{|x-y|\leq 1}\frac{|u(x)|^2}{|x-y|^\alpha}\,dx\big\}\\
	&\leq \sup_{y\in\mathbb{R}^3}\{\|u\|^2_2+\big(\int_{|x-y|\leq 1}\frac{1}{|x-y|^{3\alpha}}\,dx\big)^{1/3}\|u\|^2_3\},
	\end{align*}
	since $\alpha<1$, then
	$$\int_{|x-y|\leq 1}\frac{1}{|x-y|^{3\alpha}}\,dx<M,\ \ \text{for some $M>0$ independent of $y$.}$$
	Using the Sobolev inequality $\|u\|_3\lesssim u\|\|_{H^{1/2}}$, then it follows that \eqref{H-alpha-inf} holds.
	
	  To prove \eqref{H-alpha-inf1}, notice that
	  \begin{align*}
	  &\|\frac{1}{|x|^\alpha}\ast(|u|^2-|v|^2)\|_{\infty}=\sup_{y\in\mathbb{R}^3}\big|\int_{\mathbb{R}^3}\frac{|u(x)|^2-|v(x)|^2}{|x-y|^\alpha}\,dx\big|\\
	  &\lesssim\bigg(\sup_{y\in\mathbb{R}^3}\int_{\mathbb{R}^3}\frac{|u(x)+v(x)|^2}{|x-y|^\alpha}\,dx\bigg)^{1/2}\bigg(\sup_{y\in\mathbb{R}^3}\int_{\mathbb{R}^3}\frac{|u(x)-v(x)|^2}{|x-y|^\alpha}\,dx\bigg)^{1/2}\\
	  &\lesssim \|u+v\|_{H^{1/2}}\|u-v\|_{H^{1/2}} \ \ \ \ (\text{from \eqref{H-alpha-inf}})\\
	  &\lesssim (\|u\|_{H^{1/2}}+\|v\|_{H^{1/2}})\|u-v\|_{H^{1/2}}.
	  \end{align*}
	
	 By weak Young inequality (see \cite{Lieb})
	 \begin{align*}
	 &\|\frac{1}{|x|^\alpha}\ast (|u|^2-|v|^2)\|_{\frac{6}{\alpha}}\lesssim\|\frac{1}{|x|^\alpha}\|_{\frac{3}{\alpha},w}\||u|^2-|v|^2\|_{\frac{6}{6-\alpha}}\\
	 &\lesssim \|u+v\|_{\frac{6}{3-\alpha}}\|u-v\|_2\\
	  &\lesssim(\|u\|_{H^{1/2}}+\|v\|_{H^{1/2}})\|u-v\|_2.
	 \end{align*}
	 We notice that by the definition of Riesz potential (see \cite{Stein1970}) $\frac{C_\alpha}{|x|^\alpha}\ast f$ can be expressed as $D^{\alpha-3} f=(-\triangle)^{-\frac{3-\alpha}{2}}f$ (here $f\in \mathcal{S}(\mathbb{R})^3$ is innitially assumed, but our arguments follow by density). Thus, we have
	 \begin{align*}
	 &\|D^{1/2}\big(\frac{1}{|x|^\alpha}\ast (|u|^2-|v|^2)\big)\|_6\lesssim \|D^{1/2+\alpha-3} (|u|^2-|v|^2)\|_6\\
	 &\lesssim\|\frac{1}{|x|^{\frac{1}{2}+\alpha}}\ast (|u|^2-|v|^2)\|_6\\
	 &\lesssim\|\frac{1}{|x|^{\alpha+\frac{1}{2}}}\|_{\frac{6}{1+2\alpha},w}\||u|^2-|v|^2\|_{\frac{3}{3-\alpha}}\\
	 &\lesssim \|u+v\|_{\frac{6}{3-\alpha}}\|u-v\|_{\frac{6}{3-\alpha}}\\
	  &\lesssim (\|u\|_{H^{1/2}}+\|v\|_{H^{1/2}})\|u-v\|_{H^{1/2}}.
	 \end{align*}
\end{proof}		

\begin{lemma}\label{nonli-alpha-Lips} Suppose that $0<\alpha<1$.
	The map $J(u):=(\frac{1}{|x|^\alpha}\ast |u|^2)u$ is locally Lipschitz continuous from $H^{1/2}(\mathbb{R}^3)$ into itself with
	$$\|J(u)-J(v)\|_{H^{1/2}}\lesssim (\|u\|^2_{H^{1/2}}+\|u\|^2_{H^{1/2}})\|u-v\|_{H^{1/2}},$$
	for all $u,v\in H^{1/2}(\mathbb{R}^3)$.
\end{lemma}
\begin{proof}
 With the estimates given in Lemma \ref{estimate-h-4}, we can prove this lemma by following the similar argument as Lemma 1 in \cite{Lenzmann2007}. Now we sketch the proof.

 The same arguments as \cite{Lenzmann2007}, we know that
 $$\|u\|_2+\|D^{1/2}u\|_2\lesssim\|u\|_{H^{1/2}}\lesssim \|u\|_2+\|D^{1/2}u\|_2$$
 it is sufficient to estimate the quantities
 $$I:=\|J(u)-J(v)\|_2,\ \ \ \ II:=\|D^{1/2}[J(u)-J(v)]\|_2.$$
 The same argument as (15) in \cite{Lenzmann2007}, by H\"{o}lder inequality
 \begin{align}
 &I\lesssim \big\|\big(\frac{1}{|x|^\alpha}\ast(|u|^2-|v|^2)\big)(u+v)\big\|_2 +\big\|\big(\frac{1}{|x|^\alpha}\ast(|u|^2+|v|^2)\big)(u-v)\big\|_2 \notag\\
 &\lesssim \big\|\frac{1}{|x|^\alpha}\ast(|u|^2-|v|^2)\big\|_{\frac{6}{\alpha}}\|u+v\|_{\frac{6}{3-\alpha}} +\big\|\frac{1}{|x|^\alpha}\ast(|u|^2+|v|^2)\big\|_\infty\|u-v\|_2.
 \end{align}
 Then by \eqref{H-alpha-inf}, \eqref{H-alpha-6} and together with Sobolev inequality $\|u\|_q\lesssim \|u\|_{H^{1/2}}$ (for all $2\leq q\leq 3$) then
 $$I\lesssim (\|u\|^2_{H^{1/2}}+\|u\|^2_{H^{1/2}})\|u-v\|_{H^{1/2}}.$$

 On the other hand, as (18) in \cite{Lenzmann2007}, using Leibniz rule  we also have
 \begin{align*}
 II&\lesssim \big\|D^{1/2}\big(\frac{1}{|x|^\alpha}\ast(|u|^2-|v|^2)\big)\big\|_6\|u+v\|_3+ \big\|\frac{1}{|x|^\alpha}\ast(|u|^2-|v|^2)\big\|_\infty\|D^{1/2}(u+v)\|_2
 \\ &+\big\|D^{1/2}\big(\frac{1}{|x|^\alpha}\ast(|u|^2+|v|^2)\big)\big\|_6\|u-v\|_3+\big\|\frac{1}{|x|^\alpha}\ast(|u|^2+|v|^2)\big\|_\infty\|D^{1/2}(u-v)\|_2.
 \end{align*}
 By \eqref{H-alpha-inf1} and \eqref{H-alpha-6-1}, then
$$II\lesssim (\|u\|^2_{H^{1/2}}+\|u\|^2_{H^{1/2}})\|u-v\|_{H^{1/2}}.$$
Thus we complete the proof.
 \end{proof}

Therefore, the nonlinearity $F(u)$ is local Lipschitz continuous, by standard methods for evolution equations with locally Lipschitz nonlinearities, we have following local well-posedness theorem.
\begin{theorem}\label{local-well-p-theorem}
	Let $m>0$, $0<\alpha<1$, $\beta\in \mathbb{R}$. Then  initial value  probem \eqref{e} is locally well-posed in $H^{1/2}(\mathbb{R}^3)$. This means that, for every $\psi_0\in H^{1/2}(\mathbb{R}^3)$ , there exists a uniqueness solution
	$$u\in C^0\big([0,T);H^{1/2}(\mathbb{R}^3)\big)\cap C^1\big([0,T);H^{-1/2}(\mathbb{R}^3)\big),$$
and it depends continuously on $u_0$. Here $T\in (0,\infty]$ is the maximal time of existence , where we have that either $T=\infty$ or $T<\infty$ and $\lim_{t\uparrow T}\|u\|_{H^{1/2}}=\infty$ holds.
\end{theorem}

\subsection{Global Well-posedness}

\begin{lemma}[Conservation Laws]\label{conser-law}
	The local-in-time solutions of Theorem \ref{local-well-p-theorem} obey conservation of energy and charge, i.e.,
	\begin{align*}
	\mathcal{E}_\beta(\psi(t))=\mathcal{E}_\beta(\psi_0)\ \ \ \text{and} \ \ \ \ \|\psi(t)\|^2_2=\|\psi_0\|^2_2,
	\end{align*}
	for all $t\in [0,T)$.
\end{lemma}
\begin{proof}
	 Let $F(u)$ given in \eqref{A-fu} take the place of $F(u)$ in the proof of Lemma 2 in \cite{Lenzmann2007}, and combine Lemma \ref{nonli-alpha-Lips}, then following the same arguments as \cite{Lenzmann2007} one can easy to check this lemma.
\end{proof}\\
\textbf{The end proof of Theorem \ref{Global-well-theorem}:}\\
\begin{proof}
	By the blow-up alternative of Theorem \ref{local-well-p-theorem}, global-in-time existence follows from an priori bound the form
	\begin{align}
	\|\psi(t)\|_{H^{1/2}}\leq C(\psi_0).
	\end{align}
Notice that by Lemma \ref{conser-law}, we have
\begin{align}
\mathcal{E}_\beta(\psi(t))=\mathcal{E}_\beta(\psi_0),\ \ \ \|\psi(t)\|^2_2=\|\psi_0\|^2_2
\end{align}
Thus, it is sufficient to prove that $\langle \psi(t), \sqrt{-\triangle}\,\psi(t)\rangle $ is uniformly bounded.

 For condition (i) in Theorem \ref{Global-well-theorem}, if $\beta>0$, the same as \eqref{lower1}
\begin{align}\label{lower1-t}
\mathcal{E}_\beta(\psi_0)=\mathcal{E}_\beta(\psi(t))\geq \frac{1}{2}(1-\frac{\|\psi_0\|^2_2}{N_c})\langle\psi(t),\sqrt{-\triangle}\,\psi(t)\rangle;
\end{align}
if $\beta\leq 0$, the same as \eqref{lower2}
\begin{align}\label{lower2-t}
\mathcal{E}_\beta(\psi_0)=\mathcal{E}_\beta(\psi(t))\geq \frac{1}{2}(1-\frac{\|\psi_0\|^2_2}{N_c})\langle\psi(t),\sqrt{-\triangle}\,\psi(t)\rangle+\beta\,C_2\langle\psi(t),\sqrt{-\triangle}\,\psi(t)\rangle^{\alpha}.
\end{align}
Since $\|\psi_0\|^2_2<N_c$ and $0<\alpha<1$, the above two inequality show that $\langle \psi(t), \sqrt{-\triangle}\,\psi(t)\rangle $ is uniformly bounded. Thus we complete the global well-poseness for condition (i).

Now we go to prove for condition (ii). In this case, note that $\langle \psi(t), \sqrt{-\triangle}\,\psi(t)\rangle$ can not be controlled by $\mathcal{E}_\beta(\psi_0)$ like \eqref{lower1-t} and \eqref{lower2-t} due to $\|\psi_0\|^2_2=N_c$. To overcome it, we use the blow-up analysis and prove by contradiction.
On the contrary, we now suppose that $\langle \psi(t), \sqrt{-\triangle}\,\psi(t)\rangle $ is unbounded. Then there exists a subsequence $\psi_n:=\psi(t_n)$ with $t_n\rightarrow T$ (as $n\rightarrow\infty$), such that
\begin{align}\label{gradi-est-inft-t}
\langle \psi_n, \sqrt{-\triangle}\psi_n\rangle \rightarrow +\infty, \ \ (n\rightarrow \infty)
\end{align}
Since $\sqrt{-\triangle+m^2}\geq \sqrt{-\triangle}$, by \eqref{maininequality} and conservation laws
we have
\begin{align}\label{u-n-alpha-t} \frac{\beta}{4}\int_{\mathbb{R}^3}(\frac{1}{|x|^\alpha}\ast|\psi_n|^2)|\psi_n|^2\,dx\leq \mathcal{E}_{\beta}(\psi_n)=\mathcal{E}_{\beta}(\psi_0),
\end{align}
and also
\begin{align}\label{gradi-u-Hartree-est-t}
0\leq \frac{1}{2}\langle \psi_n, \sqrt{-\triangle}\psi_n\rangle-\frac{1}{4}\int_{\mathbb{R}^{3}}(\frac{1}{|x|}\ast|\psi_n|^2)|\psi_n|^2\,dx\leq \mathcal{E}_{\beta}(\psi_0).
\end{align}

Note that, \eqref{gradi-est-inft-t}-\eqref{gradi-u-Hartree-est-t} shows the similar results as \eqref{gradi-est-inft}-\eqref{gradi-u-Hartree-est}. Then the same arguments as \eqref{epsi-un}-\eqref{alpha-0}, one can  obtain a contradiction.
 Therefore, $\langle \psi(t), \sqrt{-\triangle}\,\psi(t)\rangle $ is uniformly bounded, we complete the theorem.
\end{proof}

\subsection{The  proof of Theorem \ref{theorem-stable}}
\begin{proof}
	Let us now assume that orbital stability (in the sense defined Theorem \ref{theorem-stable}) does not hold. Then there exists a sequence on initial date, $\{\psi_n(0)\}$, in $H^{1/2}(\mathbb{R}^3 )$ with
	\begin{align}\label{orbi-sta1}
	\inf_{\varphi \in S_{N_c}}\|\psi_n(0)-\varphi\|_{H^{1/2}}\rightarrow 0, \ \ \ \text{as $n\rightarrow \infty$},
	\end{align}
	and some $\epsilon >0$ such that
	\begin{align}\label{orbi-sta2}
	\inf_{\varphi\in S_{N_c}}\|\psi_n(t_n)-\varphi\|_{H^{1/2}}> \epsilon,\ \ \ \text{for all $n\geq 0$,}
	\end{align}
for a suitable sequence of times $\{t_n\}$. Note that \eqref{orbi-sta1} implies that $\mathcal{N}(\psi_n(0))\rightarrow N_c$ as $n\rightarrow \infty.$	
	Next consider the sequence, $\{u_n\}$, in $H^{1/2}(\mathbb{R}^3)$ that is given by
	\begin{align}
	u_n:=\psi_n(t_n).
	\end{align}
By conservation laws given in Lemma \ref{conser-law}, then $\|u_n\|^2_2=\|\psi_n(0)\|^2_2$	and $\mathcal{E}_\beta(u_n)=\mathcal{E}_\beta(\psi_n(0))$.
By \eqref{orbi-sta1}, it follows that
\begin{align}
\lim_{n\rightarrow \infty}\mathcal{E}_\beta(u_n)=E(\beta,N_c) \ \ \ \text{and} \ \ \ \ \lim_{n\rightarrow \infty}\|u_n\|^2_2=N_c.
\end{align}
Defining the rescaled sequence
$$\widetilde{u}_n:=\sqrt{\frac{N_c}{\|u_n\|^2_2}}u_n,$$
Notice that, the same arguments as  Lemma \ref{minimi-seq-bound}, $\{u_n\}$ has to be bounded in $H^{1/2}(\mathbb{R}^3)$. It  follows that
\begin{align}\label{u-n-tilde-u-n}
\|u_n-\widetilde{u}_n\|_{H^{1/2}}\leq  C\big|1-\frac{N_c}{\|u_n\|^2_2}\big|\rightarrow0.
\end{align}
Thus we deduce that
\begin{align}
\lim_{n\rightarrow \infty}\mathcal{E}_\beta(\widetilde{u}_n)=E(\beta,N_c) \ \ \ \text{and} \ \ \ \ \|\widetilde{u}_n\|^2_2=N_c,\ \ \text{for all $n$.}
\end{align}
Therefore, $\{\widetilde{u}_n\}$ is a minimizing sequence for problem \eqref{minimizing-problem}. By Theorem \ref{theorem1}, for $\beta$ small enough, this sequence has to contain a subsequence, $\{\widetilde{u}_{n_k}\}$, that strongly converges in $H^{1/2}(\mathbb{R}^3)$ to some minimizer $\varphi\in S_{N_c}$. In particular, inequality \eqref{orbi-sta2} cannot hold when $u_n=\psi_n(t_n)$ is replaced by $\widetilde{u}_n$. However, in view of \eqref{u-n-tilde-u-n}, then inequality \eqref{orbi-sta2} cannot hold for $\{u_n\}$ itself. Thus, we are led to a contradiction and the proof of  Theorem \ref{theorem-stable} is complete.
\end{proof}

\section{Asymptotic analysis of minimizers for $E(\beta,N_c)$ as $\beta\rightarrow 0^+$}
\subsection{Energy estimates}
\begin{lemma}\label{lemma-optimal-est}
Under the assimptions of Theorem \ref{theorem1}, there exist two constants $C_1>0$ and $C_2>0$, independent of $\beta$ such that as $\beta\rightarrow 0^+$
\begin{align}\label{optimal-energy-estimates}
C_1\beta^{\frac{1}{\alpha+1}}\leq E(\beta,N_c) \leq C_2\beta^{\frac{1}{\alpha+1}}.
\end{align}
\end{lemma}
\begin{proof}
	The upper bound follows \eqref{upper-bound} by using the same test function. Thus we only need to prove the lower bound.
	
	Now let $\varphi_\beta$ be a nonnegative minimizer of \eqref{minimizing-problem}, first we claim that \\
\textbf{Claim 1:}  $\langle\varphi_\beta,\sqrt{-\triangle}\varphi_\beta\rangle\rightarrow +\infty$ as $\beta\rightarrow 0$.

In fact, if this claim is not true, then  there exists a subsequence $\{\beta_k\}$ with $\beta_k\rightarrow 0$ as
$k\rightarrow\infty$, such that $\{\varphi_{\beta_k}\}$ is uniformly bound in $H^{1/2}(\mathbb{R}^3)$. Hence there exists $\varphi_0\in H^{1/2}(\mathbb{R}^3)$ and a weakly converging subsequence, still denoted by $\{\varphi_{\beta_k}\}$, such that
\begin{align*}
\varphi_{\beta_k}\rightharpoonup\varphi_0,\ \ \text{weakly in $H^{1/2}(\mathbb{R}^3)$}.
\end{align*}
(1)\textit{Vanishing does not occur}\\
If vanishing occurs, it follows from Lemma \ref{vanishing-alpha} we have
\begin{align*}
\lim_{n\rightarrow \infty}\int_{\mathbb{R}^3}(\frac{1}{|x|^\alpha}\ast|\varphi_{\beta_k}|^2)|\varphi_{\beta_k}|^2\,dx=0;\ \ \lim_{n\rightarrow \infty}\int_{\mathbb{R}^3}(\frac{1}{|x|}\ast|\varphi_{\beta_k}|^2)|\varphi_{\beta_k}|^2\,dx=0.
\end{align*}
Since $\sqrt{-\triangle+m^2}-m\geq 0$, it follows that
\begin{align*}
&\liminf_{k\rightarrow\infty}E(\beta_k,N_c)=\liminf_{k\rightarrow\infty}\mathcal{E}_{\beta_k}(\varphi_{\beta_k})
= \liminf_{k\rightarrow\infty} \big\{\frac{1}{2}\langle\varphi_{\beta_k},(\sqrt{-\triangle+m^2})\varphi_{\beta_k}\rangle\big\}\\&\geq \liminf_{k\rightarrow\infty}\frac{1}{2}m\|\varphi_{\beta_k}\|^2_2=\frac{1}{2}mN_c.
\end{align*}
On the other hand, note that by the upper bound of \eqref{optimal-energy-estimates} we know that for $\beta$ small enough, we have $C_2\beta^{\frac{1}{\alpha+1}}<\frac{1}{2}mN_c$, this means that
$$\limsup_{k\rightarrow\infty}E(\beta_k,N_c)< \frac{1}{2}mN_c.$$
Thus, we obtain a contradiction. Therefore, vanishing does not occur.\\
(2)\textit{Dichotomy does not occur:} \\
If the dichotomy occurs, by Lemma \ref{concen-compact}(iii) below, then there exists $\lambda\in (0,N_c)$ such that, for every $\epsilon>0$, there exists  two bounded dichotomy subsequences in $H^{1/2}(\mathbb{R}^3)$ denoted by $\{\varphi^1_{\beta_k}\}$ and $\{\varphi^2_{\beta_k}\}$ with
\begin{align}\label{l-e-beta}
\lambda-\epsilon\leq \|\varphi^1_{\beta_k}\|^2_2\leq \lambda+\epsilon,\ \ (N_c-\lambda)-\epsilon\leq \|\varphi^2_{\beta_k}\|^2_2\leq  (N_c-\lambda)+\epsilon
\end{align}
for $k$ sufficiently large. Moreover,  \eqref{binding-est1} and \eqref{binding-est2} allow us to deduce that, there exists $r_1(k)$ and $r_2(\epsilon)$ such that
\begin{align}\label{energy-dich-ineq}
\mathcal{E}_{\beta_k}(\varphi_{\beta_k})-\mathcal{E}_{\beta_k}(\varphi^1_{\beta_k})-\mathcal{E}_{\beta_k}(\varphi^2_{\beta_k})\geq -r_1(k)-r_2(\epsilon),
\end{align}
where $r_1(k)\rightarrow 0$ as $k\rightarrow\infty$, and
$r_2(\epsilon)\rightarrow 0$ as $\epsilon\rightarrow 0$. Then
\begin{align*}
&0=E(0,N_c)=\lim_{k\rightarrow\infty}\mathcal{E}_{\beta_k}(\varphi_{\beta_k})\\
&\geq \liminf_{k\rightarrow\infty}\{\mathcal{E}_{\beta_k}(\varphi^1_{\beta_k})+\mathcal{E}_{\beta_k}(\varphi^2_{\beta_k})-r_1(k)\}-r_2(\epsilon)\ \ \ \ \ \text{(by \eqref{energy-dich-ineq})}\\
&\geq\liminf_{k\rightarrow\infty}\{\mathcal{E}_{0}(\varphi^1_{\beta_k})+\mathcal{E}_{0}(\varphi^2_{\beta_k})\}-r_2(\epsilon)\ \ \ \ \ \text{(since $\beta_k>0$)}\\
&\geq\liminf_{k\rightarrow\infty}\{E(0,\|\varphi^1_{\beta_k}\|^2_2)+E(0,\|\varphi^2_{\beta_k}\|^2_2)\}-r_2(\epsilon)\\
&\geq E(0,\lambda+\epsilon)+E(0,(N_c-\lambda)+\epsilon)-r_2(\epsilon),
\end{align*}
the last inequality comes from \eqref{l-e-beta} and the fact that  $E(0,N)$ is decreasing in $N$, see Lemma \ref{less-zero} (i) above.
Passing to the limit $\epsilon\rightarrow0$ and by continuity of $E(\beta,N)$ in $N$, we deduce that
\begin{align}
E(0,N_c)\geq E(0,\lambda)+E(0,N_c-\lambda)
\end{align}
holds for some $0<\lambda<N_c$. This contradicts the strict subadditivity condition \eqref{subadd-ineq} with $\beta=0$.

Therefore, compactness happens. The same arguments as \cite{F-J-Lenzmann2007}, then up to a translation, the  sequence $\{\varphi_{\beta_k}\}$ are relatively compact in $H^{1/2}(\mathbb{R}^3)$. Thus there exists $\{y_k\}$ with $y_k\in \mathbb{R}^3$ such that $\widetilde{\varphi}_{\beta_k}=\varphi_{\beta_k}(x+y_k)$ satisfies
\begin{align}
\widetilde{\varphi}_{\beta_k}\rightarrow \varphi_0, \ \ \text{strongly in $H^{1/2}(\mathbb{R}^3)$.}
\end{align}
This  implies that $\lim_{k\rightarrow\infty}\mathcal{E}_{\beta_k}(\widetilde{\varphi}_{\beta_k})=\mathcal{E}_{0}(\varphi_0).$ It follows that
\begin{align}
E(0,N_c)\leq \mathcal{E}_{0}(\varphi_0)=\lim_{k\rightarrow\infty}\mathcal{E}_{\beta_k}(\widetilde{\varphi}_{\beta_k})=\lim_{k\rightarrow\infty}\mathcal{E}_{\beta_k}(\varphi_{\beta_k})=E(0,N_c)
\end{align}
This implies that $\varphi$ is a minimizer of $E(0,N_c)$, contradicting that $E(0,N_c)$ has no minimizer (see Theorem \ref{theorem1} (i) above). Thus we conclude that $\langle\varphi_\beta,\sqrt{-\triangle}\varphi_\beta\rangle\rightarrow +\infty$ as $\beta\rightarrow 0^+$.

Now define
\begin{align}\label{epsilon-beta}
\epsilon=\langle\varphi_\beta,\sqrt{-\triangle}\varphi_\beta\rangle^{-1},
\end{align}
then $\epsilon\rightarrow 0$ as $\beta\rightarrow 0$. Note that by \eqref{maininequality} with the fact $\sqrt{-\triangle+m^2}\geq \sqrt{-\triangle}$, and the upper bound of \eqref{lemma-optimal-est}, we have
\begin{align*}
0&\leq \frac{1}{2}\langle \varphi_{\beta}, \sqrt{-\triangle}\varphi_{\beta}\rangle-\frac{1}{4}\int_{\mathbb{R}^{3}}(\frac{1}{|x|}\ast|\varphi_{\beta}|^2)|\varphi_{\beta}|^2\,dx\\
&\leq E(\beta,N_c)\leq C_2 \beta^{1+\alpha}\rightarrow 0.
\end{align*}
this implie that as $\beta\rightarrow 0$
\begin{align}\label{psi-beta-Hartree}
\frac{\epsilon}{2}\int_{\mathbb{R}^{3}}(\frac{1}{|x|}\ast|\varphi_{\beta}|^2)|\varphi_{\beta}|^2\,dx\rightarrow 1.
\end{align}
The same arguments as \eqref{epsi-un}-\eqref{w-n-below-estimate}, there exist $y_\epsilon\in \mathbb{R}^3$ and $R_0>0,\eta>0$, and define
\begin{align}\label{w-beta}
w_{\beta}(x):=\epsilon^{\frac{3}{2}} \varphi_{\beta}(\epsilon x +\epsilon y_{\epsilon}),
\end{align}
then $w_\beta$ satisfies
\begin{align}\label{w-beta-bound-1}
\|w_\beta\|^2_2=N_c,\  \ \ \langle w_\beta, \sqrt{-\triangle}w_\beta\rangle=1,\ \ \
M\leq \int_{\mathbb{R}^{3}}(\frac{1}{|x|}\ast|w_\beta|^2)|w_\beta|^2\,dx \leq \frac{1}{M},
\end{align}
for some $M>0$,
such that
\begin{align}\label{w-beta-below-estimate}
\liminf_{\epsilon\rightarrow 0}\int_{B(0, R_0)}|w_{\beta}(x)|^2\,dx\geq \eta >0.
\end{align}
\textbf{Claim 2:} there exists a subsequence $\{\beta_k\}$ with $\beta_k\rightarrow 0$ such that
$w_{\beta_k}\rightarrow w_0$ strongly in $L^q(\mathbb{R}^3)$ for all $2\leq q<3$ where $w_0$ satisfies
\begin{align}\label{w-0-fom}
w_0=\gamma^{\frac{3}{2}}Q(\gamma|x-y_0|)
\end{align}
for some $y_0\in \mathbb{R}^3$, $\gamma>0$, and $Q$ satisfies \eqref{optimal-equ}, we also have $\|w_0\|^2_2=N_c$.

Since $\varphi_\beta$ is a minimizer of  \eqref{minimizing-problem}, it satisfies the Euler-Lagrange equation \eqref{Eular-Lagrange}, that is  \begin{align}\label{Eular-Lagrange1}
\sqrt{-\triangle+m^2}\,\varphi_\beta+\big((\frac{\beta}{|x|^\alpha}-\frac{1}{|x|})\ast |\varphi_\beta|^2\big)\varphi_\beta =\mu_\beta\, \varphi_\beta,
\end{align}
for $\mu_\beta\in\mathbb{R}$, which is a suitable Lagrange multiplier. In fact,
\begin{align}\label{mu-beta}
\mu_\beta =\frac{1}{N_c}\bigg\{E(\beta,N_c)-\frac{1}{2}\int_{\mathbb{R}^3}(\frac{1}{|x|}\ast|\varphi_\beta|^2)|\varphi_\beta|^2\,dx+
\frac{\beta}{2}\int_{\mathbb{R}^3}(\frac{1}{|x|^\alpha}\ast|\varphi_\beta|^2)|\varphi_\beta|^2\,dx\bigg\}
\end{align}
Then $w_\beta(x)$ defined in \eqref{w-beta} satisfies
\begin{align}\label{w-beta-equation-1}
\sqrt{-\triangle+\epsilon^2m^2}\,w_\beta+\beta\epsilon^{1-\alpha}(\frac{1}{|x|^\alpha}\ast|w_\beta|^2)w_\beta-(\frac{1}{|x|}\ast|w_\beta|^2)w_\beta =\epsilon\mu_\beta\, w_\beta.
\end{align}
Note that
By \eqref{maininequality} and the upper bound of (\ref{optimal-energy-estimates}) we have
$$\frac{\beta}{4}\int_{\mathbb{R}^3}(\frac{1}{|x|^\alpha}\ast|\psi_\beta|^2)|\psi_\beta|^2\,dx\leq E(\beta,N_c)\leq C_2\beta^{\frac{1}{1+\alpha}}.$$
 Then combining \eqref{mu-beta} and  (\ref{w-beta-bound-1}) we know that $\epsilon \mu_\beta$ is uniformly bounded and strictly negative for $\beta$ close to $0$. Passing to a subsequence $\{\beta_k\}$, we have $\lim_{\beta_k\rightarrow 0^+}\epsilon \mu_{\beta_k}= -\gamma<0$, where $\gamma >0$, and $w_{\beta_k}\rightharpoonup w_0$ in $H^{1/2}(\mathbb{R}^3)$ for some $w_0\in H^{1/2}(\mathbb{R}^3)$ such that $w_0\geq 0$.  Passing to weak limit in (\ref{w-beta-equation-1}), then $w_0$ satisfies
\begin{align}
\sqrt{-\triangle} w_0(x) -(\frac{1}{|x|}\ast|w_0|^2)w_0=-\gamma w_0(x).
\end{align}
Moreover, it follows from \eqref{w-beta-below-estimate} that $w_0\not\equiv 0$. Let $w^{\gamma}_0=\gamma w_0(\gamma x)$, then $w^{\gamma}_0$ satisfies equation \eqref{optimal-equ} such that $\|w^{\gamma}_0\|^2_2=\|w_0\|^2_2$. Moreover, by Pohozaev identity,
$$
\langle w^{\gamma}_0, \sqrt{-\triangle} \,w^{\gamma}_0\rangle=\int_{\mathbb{R}^3}|w^{\gamma}_0|^2\,dx=\frac{1}{2}\int_{\mathbb{R}^3}(\frac{1}{|x|}\ast|w^{\gamma}_0|^2)|w^{\gamma}_0|^2\,dx.
$$
Then by \eqref{best-constant}
\begin{align}\label{w-0-nc}
\frac{N_c}{2}=\frac{1}{S}\leq \frac{\langle w^{\gamma}_0,\sqrt{-\triangle}\,w^{\gamma}_0\rangle\,\langle w^{\gamma}_0,w^{\gamma}_0\rangle}{\int_{\mathbb{R}^3}(\frac{1}{|x|}\ast|w^{\gamma}_0|^2)|w^{\gamma}_0|^2\,dx}=\frac{1}{2}\|w^{\gamma}_0\|^2_2=\frac{1}{2}\|w_0\|^2_2.
\end{align}
Note that since $w_{\beta_k}\rightharpoonup w_0$ in $H^{1/2}(\mathbb{R}^3)$, by Fauto lemma, then
$$\|w_0\|^2_2\leq \liminf \|w_\beta\|^2_2=N_c.$$
Combining \eqref{w-0-nc} we have $\|w_0\|^2_2=N_c$, this implies that $w_{\beta_k}$ converges to  $ w_0$ strongly in $L^2(\mathbb{R}^3)$. Since $w_{\beta_k}$ is uniformly bounded in $H^{1/2}(\mathbb{R}^3)$, it follows that $w_{\beta_k}\rightarrow w_0$ in $L^{q}(\mathbb{R}^3)$ for $q\in [2,3)$. Thus  we complete the proof of \textbf{Claim 2}.

Note that
\begin{align*}
&E(\beta_k,N_c)=\frac{1}{2}\langle w_{\beta_k},\frac{\sqrt{-\triangle+\epsilon^2 m^2} }{\epsilon}w_{\beta_k}\rangle-\frac{1}{4\epsilon}\int_{\mathbb{R}^3}(\frac{1}{|x|}\ast|w_{\beta_k}|^2)|w_{\beta_k}|^2\,dx\\
&+\frac{\beta_k}{4\epsilon^\alpha}\int_{\mathbb{R}^3}(\frac{1}{|x|^\alpha}\ast|w_{\beta_k}|^2)|w_{\beta_k}|^2\,dx\\
&\geq \frac{1}{2}\langle w_{\beta_k},\frac{\sqrt{-\triangle+\epsilon^2 m^2}-\sqrt{-\triangle} }{\epsilon}w_{\beta_k}\rangle+\frac{\beta_k}{4\epsilon^\alpha}\int_{\mathbb{R}^3}(\frac{1}{|x|^\alpha}\ast|w_{\beta_k}|^2)|w_{\beta_k}|^2\,dx\ \ \ \text{(by \eqref{maininequality})}\\
&=\frac{\epsilon}{2}\langle w_{\beta_k},\frac{m^2}{\sqrt{-\triangle+\epsilon^2 m^2}+\sqrt{-\triangle} }w_{\beta_k}\rangle+\frac{\beta_k}{4\epsilon^\alpha}\int_{\mathbb{R}^3}(\frac{1}{|x|^\alpha}\ast|w_{\beta_k}|^2)|w_{\beta_k}|^2\,dx\\
&\geq \frac{\epsilon}{4}\langle w_{\beta_k},\frac{m^2}{\sqrt{-\triangle+1} }w_{\beta_k}\rangle+\frac{\beta_k}{4\epsilon^\alpha}\int_{\mathbb{R}^3}(\frac{1}{|x|^\alpha}\ast|w_{\beta_k}|^2)|w_{\beta_k}|^2\,dx\ \ \ \text{(for $\epsilon$ small enough)}
\end{align*}
By \textbf{Claim 2} we know that $w_{\beta_k}\rightarrow w_0$ strongly in $L^q(\mathbb{R}^3)$ where $w_0$ is given in \eqref{w-0-fom}, then there exist constants $M_1>0$ and $M_2>0$, independent of $\beta_k$ such that for $\beta_k$ small enough
$$\langle w_{\beta_k},\frac{m^2}{\sqrt{-\triangle+1} }w_{\beta_k}\rangle\geq M_1,\ \ \ \int_{\mathbb{R}^3}(\frac{1}{|x|^\alpha}\ast|w_{\beta_k}|^2)|w_{\beta_k}|^2\,dx\geq M_2,$$
it follows that
$$E(\beta,N_c)\geq \frac{\epsilon}{2}M_1+\frac{\beta_k}{4\epsilon^\alpha}M_2\geq C_1\beta^{\frac{1}{1+\alpha}}_k.$$
Therefore, we get the lower bound.
\end{proof}
\subsection{The proof of Theorem \ref{theorm-asymp-beta}}
\begin{lemma}
Under the assumptions of Theorem \ref{theorem1},  let $\varphi_{\beta}(x)$ be a minimizer of $E(\beta,N_c)$. Then there exist positive constants $K_1$, $K_2$, independent of $\beta$, such that as $\beta\rightarrow 0^+$
	\begin{align}\label{psi-alpha-lambda-est}
	K_1\beta^{-\,\frac{\alpha}{1+\alpha}}\leq \int_{\mathbb{R}^3}(\frac{1}{|x|^\alpha}\ast|\varphi_\beta|^2)|\varphi_\beta|^2\,dx\leq K_2 \beta^{-\,\frac{\alpha}{1+\alpha}}.
	\end{align}
\end{lemma}
\begin{proof}
	The upper bound  follows from (\ref{optimal-energy-estimates}).
	
	To prove the lower bound, we choose a $\beta_1$ satisfying that $ \beta_1=\theta\beta$ with $\theta>1$. We have
	\begin{align*}
	&E(\beta_1,N_c)\leq \mathcal{E}_{\beta_1 }(\varphi_{\beta})\\
	&=E(\beta,N_c)+\frac{\beta_1 -\beta}{4}\int_{\mathbb{R}^3}(\frac{1}{|x|^\alpha}\ast|\varphi_\beta|^2)|\varphi_\beta|^2\,dx,
	\end{align*}
	then
	\begin{align*}
	 \int_{\mathbb{R}^3}(\frac{1}{|x|^\alpha}\ast|\varphi_\beta|^2)|\varphi_\beta|^2\,dx\geq \frac{C_1 \beta^{\frac{1}{1+\alpha}}_1 -C_2 \beta^{\frac{1}{1+\alpha}}}{\beta_1-\beta}
	 =\frac{C_2\beta^{-\,\frac{\alpha}{1+\alpha}}(\frac{C_1}{C_2}\,\theta^{\frac{1}{1+\alpha}}-1)}{\theta-1}.
	\end{align*}
	Taking $\theta$ large enough we can get the lower bound.
\end{proof}
\vspace{0.2cm}

The following lemma is to obtain the optimal estimates for first term and last term of $E_{\beta,a^*}(u_{\lambda,a^*})$. The upper bound is harder to come by, we should employ  scaling arguments and prove by contradiction, this is quite different from the mentioned papers.
\begin{lemma}\label{estimate-u-gra-u}
	Under the assumptions of Theorem \ref{theorem1},  let $\varphi_{\beta}(x)$ be a minimizer of $E(\beta,N_c)$. Then there exist positive constants $L_1$, $L_2$, $L_3$ and $L_4$, independent of $\beta$, such that for all $\beta>0$
	\begin{align}\label{triangle-half-psi-beta-est}
	L_1\beta^{-\,\frac{1}{1+\alpha}}\leq \langle\varphi_\beta,\sqrt{-\triangle}\varphi_\beta\rangle\leq L_2 \beta^{-\,\frac{1}{1+\alpha}}.
	\end{align}
	and
	\begin{align}\label{psi-beta-hartree-est}
	L_3\beta^{-\,\frac{1}{1+\alpha}}\leq \int_{\mathbb{R}^3}(\frac{1}{|x|}\ast|\varphi_\beta|^2)|\varphi_\beta|^2\,dx\leq L_4 \beta^{-\,\frac{1}{1+\alpha}}.
	\end{align}
\end{lemma}
\begin{proof}
	Note that by Gagliardo-Nirenberg inequality \eqref{maininequality} and (\ref{optimal-energy-estimates}) we have
	 $$0\leq\frac{1}{2}\langle\varphi_\beta,\sqrt{-\triangle}\varphi_\beta\rangle-\frac{1}{4}\int_{\mathbb{R}^3}(\frac{1}{|x|}\ast|\varphi_\beta|^2)|\varphi_\beta|^2\,dx\leq E(\beta, N_c)\rightarrow 0 \ \ \ \text{as $\beta\rightarrow 0^+$}.$$
	Then
	 $$\langle\varphi_\beta,\sqrt{-\triangle}\varphi_\beta\rangle\rightarrow\frac{1}{2}\int_{\mathbb{R}^3}(\frac{1}{|x|}\ast|\varphi_\beta|^2)|\varphi_\beta|^2\,dx\ \ \ \text{as $\beta\rightarrow 0^+$}.$$
	Therefore, it suffices to prove one of  \eqref{triangle-half-psi-beta-est} and \eqref{psi-beta-hartree-est}. Next we prove \eqref{triangle-half-psi-beta-est}.
	
	First we prove the lower bound of (\ref{triangle-half-psi-beta-est}).
	Since $\|\varphi_{\beta}\|^2_2=N_c$, by Hardy-Littlewood Sobolev inequality, interpolation inequality and Sobolev inequality, we have
	\begin{align*}
	 \int_{\mathbb{R}^3}(\frac{1}{|x|^\alpha}\ast|\varphi_\beta|^2)|\varphi_\beta|^2\,dx\leq C_1\|\varphi_\beta\|^{4}_{\frac{12}{6-\alpha}}\leq C_1 \|\varphi_\beta\|^{2\alpha}_{3} \|\varphi_\beta\|^{2(2-\alpha)}_{2}\leq C_2\langle\varphi_\beta,\sqrt{-\triangle}\,\varphi_\beta\rangle^{\alpha}.
	\end{align*}
	Combing with (\ref{psi-alpha-lambda-est}) it follows that
	\begin{align*}
	\langle\varphi_\beta,\sqrt{-\triangle}\varphi_\beta\rangle\geq C_3\big(\int_{\mathbb{R}^3}(\frac{1}{|x|^\alpha}\ast|\varphi_\beta|^2)|\varphi_\beta|^2\,dx\big)^{\frac{1}{\alpha}}\geq L_1(\beta^{-\frac{\alpha}{1+\alpha}})^{1/\alpha}=L_1\beta^{-\frac{1}{1+\alpha}}
	\end{align*}
	Thus, the lower bound of (\ref{triangle-half-psi-beta-est}) holds.
	
	Now we prove the upper bound.  Define
	\begin{align}
	\epsilon_1:=\beta^{\frac{1}{1+\alpha}},\ \ w_{\beta}(x):=\epsilon^{3/2}_1 \varphi_\beta (\epsilon_1 x).
	\end{align}
	To prove the upper bound, it suffices  to prove the following fact.
	\begin{align}
	\langle w_\beta,\sqrt{-\triangle}w_\beta\rangle\leq L_2.
	\end{align}
	On the contrary, up to a subsequence  we may assume that as $\beta\rightarrow 0^+$
	\begin{align}
	\langle w_\beta,\sqrt{-\triangle}w_\beta\rangle\rightarrow \infty.
	\end{align}
	Now let
	\begin{align}
	\epsilon^{-1}_2=\langle w_\beta,\sqrt{-\triangle}w_\beta\rangle,\ \ \ \widetilde{w}_\beta(x)=\epsilon^{3/2}_2 w_\beta(\epsilon_2 x)=(\epsilon_2\epsilon_1)^{3/2} \varphi_{\beta}(\epsilon_2 \epsilon_1 x).
	\end{align}
	then $\epsilon_2\rightarrow 0$ as $\beta\rightarrow 0^+$  such that
	\begin{align*}
	\langle \widetilde{w}_\beta,\sqrt{-\triangle}\widetilde{w}_\beta\rangle=1.
	\end{align*}
	Since $\psi_\beta$ is a minimizer of $E(\beta, N_c)$, by Gagliardo-Nirenberg inequality \eqref{maininequality}, (\ref{optimal-energy-estimates}) we have
	\begin{align*}
	 0&\leq\frac{1}{2}\langle\widetilde{w}_\beta,\sqrt{-\triangle}\widetilde{w}_\beta\rangle-
	 \frac{1}{4}\int_{\mathbb{R}^3}(\frac{1}{|x|}\ast|\widetilde{w}_\beta|^2)|\widetilde{w}_\beta|^2\,dx\\
	 &=(\epsilon_1\epsilon_2)\big[\frac{1}{2}\langle\varphi_\beta,\sqrt{-\triangle}\varphi_\beta\rangle-\frac{1}{4}\int_{\mathbb{R}^3}(\frac{1}{|x|}\ast|\varphi_\beta|^2)|\varphi_\beta|^2\,dx\big]\\
	&\leq (\epsilon_1\epsilon_2) E(\beta, N_c)\rightarrow 0\ \ \ \text{as $\lambda\rightarrow 0$}.
	\end{align*}
	Then we conclude that there exists a constant $K>0$ independent of $\beta$
	\begin{align}
	 \langle\widetilde{w}_\beta,\sqrt{-\triangle}\widetilde{w}_\beta\rangle=1,\ \ \ \frac{1}{K}\leq \int_{\mathbb{R}^3}(\frac{1}{|x|}\ast|\widetilde{w}_\beta|^2)|\widetilde{w}_\beta|^2\,dx\leq K.
	\end{align}
	The same arguments as \eqref{epsi-un}-\eqref{w-n-below-estimate1}, there exist a sequence $y_\beta\in \mathbb{R}^3$  and positive constant  $\eta_0$  such that
	\begin{align}\label{w-beta-1}
	\overline{w}_{\beta}(x)=\widetilde{w}_{\beta}(x+ y_{\beta})=(\epsilon_2 \epsilon_1)^{3/2} \varphi_{\beta}(\epsilon_2 \epsilon_1 x+\epsilon_2 \epsilon_1 y_\beta)
	\end{align}
satisfies
	\begin{align}\label{w-beta-below-estimate-1}
	 \int_{\mathbb{R}^3}(\frac{1}{|x|^\alpha}\ast|\overline{w}_\beta|^2)|\overline{w}_\beta|^2\,dx\geq \eta_0 >0.
	\end{align}
	
	On the other hand, notice that $\epsilon_1 =\beta^{\frac{1}{1+\alpha}}$,
	by (\ref{w-beta-1}) and the upper bound of (\ref{psi-alpha-lambda-est}) we have
	\begin{align*}
	 \int_{\mathbb{R}^3}(\frac{1}{|x|^\alpha}\ast|\overline{w}_\beta|^2)|\overline{w}_\beta|^2\,dx&=(\epsilon_2 \epsilon_1)^{\alpha}\int_{\mathbb{R}^3}(\frac{1}{|x|}\ast|\psi_\beta|^2)|\psi_\beta|^2\,dx\\
	&\leq K_2 (\epsilon_2 \epsilon_1)^{\alpha}\beta^{-\,\frac{\alpha}{1+\alpha}}\\
	&= K_2 (\epsilon_2)^{\alpha}.
	\end{align*}
	Since $0<\alpha<1$  it follows that as $\epsilon_2\rightarrow 0$
	 $$\int_{\mathbb{R}^3}(\frac{1}{|x|^\alpha}\ast|\overline{w}_\beta|^2)|\overline{w}_\beta|^2\,dx\rightarrow 0,$$
	which contradicts (\ref{w-beta-below-estimate-1}). Therefore, the upper bound of (\ref{triangle-half-psi-beta-est}) holds.
\end{proof}\vspace{0.2cm}\\
\textbf{The end proof of Theorem \ref{theorm-asymp-beta}}: Since $\varphi_\beta$ is a non-negative minimizer of (\ref{minimizing-problem}), it satisfies the Euler-Lagrange equation \eqref{Eular-Lagrange1}.
Define
\begin{align}
\epsilon:=\beta^{\frac{1}{1+\alpha}},\ \ w_{\beta}(x):=\epsilon^{3/2} \varphi_{\beta} (\epsilon x).
\end{align}
Then $\epsilon\rightarrow 0^+$ as $\beta\rightarrow 0^+$, and $w_\beta(x)$  satisfies
\begin{align}\label{w-beta-equation}
\sqrt{-\triangle+\epsilon^2m^2}\,w_\beta+\beta\epsilon^{1-\alpha}(\frac{1}{|x|^\alpha}\ast|w_\beta|^2)w_\beta-(\frac{1}{|x|}\ast|w_\beta|^2)w_\beta =\epsilon\mu_\beta\, w_\beta.
\end{align}
Note that by (\ref{triangle-half-psi-beta-est}) and (\ref{psi-beta-hartree-est})
\begin{align}\label{gradi-w-hartree-bound}
L_1\leq\langle w_\beta,\sqrt{-\triangle}\,w_\beta\rangle\leq L_2,\ \ \
L_3\leq\int_{\mathbb{R}^3}(\frac{1}{|x|}\ast|w_\beta|^2)|w_\beta|^2\,dx \leq L_4.
\end{align}
Then, the same arguments as the proof of \textbf{Claim 2} in Lemma \ref{lemma-optimal-est},
there exists a subsequence $\{\beta_k\}$ with $\beta_k\rightarrow 0^+$ such that
\begin{align}\label{w-0-fom1}
w_{\beta_k}\rightarrow w_0=\gamma^{\frac{3}{2}}Q(\gamma|x-y_0|)\ \ \  \text{strongly in $L^q(\mathbb{R}^3)$ for all $2\leq q<3$.}
\end{align}
where $y_0\in \mathbb{R}^3$, $\gamma>0$ (will be given below), and $Q$ satisfies \eqref{optimal-equ}, and we also have $\|w_0\|^2_2=N_c$.

Note that
\begin{align*}
&E(\beta_k,N_c)=\frac{1}{2}\langle w_{\beta_k},\frac{\sqrt{-\triangle+\epsilon^2 m^2} }{\epsilon}w_{\beta_k}\rangle-\frac{1}{4\epsilon}\int_{\mathbb{R}^3}(\frac{1}{|x|}\ast|w_{\beta_k}|^2)|w_{\beta_k}|^2\,dx\\
&+\frac{\beta_k}{4\epsilon^\alpha}\int_{\mathbb{R}^3}(\frac{1}{|x|^\alpha}\ast|w_{\beta_k}|^2)|w_{\beta_k}|^2\,dx\\
&\geq \frac{1}{2}\langle w_{\beta_k},\frac{\sqrt{-\triangle+\epsilon^2 m^2}-\sqrt{-\triangle} }{\epsilon}w_{\beta_k}\rangle+\frac{\beta_k}{4\epsilon^\alpha}\int_{\mathbb{R}^3}(\frac{1}{|x|^\alpha}\ast|w_{\beta_k}|^2)|w_{\beta_k}|^2\,dx\ \ \ \text{(by \eqref{maininequality})}\\
&=\frac{\epsilon}{2}\langle w_{\beta_k},\frac{m^2}{\sqrt{-\triangle+\epsilon^2 m^2}+\sqrt{-\triangle} }w_{\beta_k}\rangle+\frac{\beta_k}{4\epsilon^\alpha}\int_{\mathbb{R}^3}(\frac{1}{|x|^\alpha}\ast|w_{\beta_k}|^2)|w_{\beta_k}|^2\,dx\\
&= \frac{\beta^{\frac{1}{1+\alpha}}_k}{2}\langle w_{\beta_k},\frac{m^2}{\sqrt{-\triangle+\epsilon^2 m^2}+\sqrt{-\triangle} }w_{\beta_k}\rangle+\frac{\beta^{\frac{1}{1+\alpha}}_k}{4}\int_{\mathbb{R}^3}(\frac{1}{|x|^\alpha}\ast|w_{\beta_k}|^2)|w_{\beta_k}|^2\,dx\\
&\ \ \ \ \ \ \ \ \ \ \ \ \ \ \ \ \ \ \ \ \ \ \ \ \ \ \ \ \ \ \ \ \ \ \ \ \ \ \ \ \ \ \ \ \ \ \ \ \ \ \ \ \ \ \ \ \ \ \ \ \ \ \ \ \ \ \text{(since $\epsilon=\beta^{\frac{1}{1+\alpha}}_k$)}.
\end{align*}
Since $w_{\beta_k}\rightarrow w_0$ strongly in $L^q(\mathbb{R}^3)$ for all $2\leq q<3$, then
\begin{align*}
&\liminf_{\beta_k\rightarrow 0}\frac{E(\beta_k,N_c)}{\beta^{\frac{1}{1+\alpha}}_k}\\
&=\liminf_{\beta_k\rightarrow 0}\frac{1}{2}\langle w_{\beta_k},\frac{m^2}{\sqrt{-\triangle+\epsilon^2 m^2}+\sqrt{-\triangle} }w_{\beta_k}\rangle+\frac{1}{4}\int_{\mathbb{R}^3}(\frac{1}{|x|^\alpha}\ast|w_{\beta_k}|^2)|w_{\beta_k}|^2\,dx\\
&\geq  \frac{1}{4}\langle w_{0},\frac{m^2}{\sqrt{-\triangle} }w_{0}\rangle+\frac{1}{4}\int_{\mathbb{R}^3}(\frac{1}{|x|^\alpha}\ast|w_{0}|^2)|w_{0}|^2\,dx\\
&=\frac{1}{4\gamma}\langle Q,\frac{m^2}{\sqrt{-\triangle} }Q\rangle+\frac{\gamma^\alpha}{4}\int_{\mathbb{R}^3}(\frac{1}{|x|^\alpha}\ast|Q|^2)|Q|^2\,dx\ \ \ \ \ \ \text{(by \eqref{w-0-fom1})}\\
&\geq \frac{1}{2}\langle Q,\frac{m^2}{\sqrt{-\triangle} }Q\rangle^{\frac{\alpha}{1+\alpha}}\bigg(\int_{\mathbb{R}^3}(\frac{1}{|x|^\alpha}\ast|Q|^2)|Q|^2\,dx\bigg)^\frac{1}{1+\alpha}.
\end{align*}
The last inequality can be obtained by taking the minimum over $\beta>0$ which is achieved by
\begin{align}\label{gamma}
\gamma=\bigg(\frac{\langle Q,\frac{m^2}{\sqrt{-\triangle} }Q\rangle}{\int_{\mathbb{R}^3}(\frac{1}{|x|^\alpha}\ast|Q|^2)|Q|^2\,dx}\bigg)^{\frac{1}{1+\alpha}}.
\end{align}

On the other hand, taking the test function
$$\psi(x)=\big(\frac{\gamma}{\beta^{\frac{1}{1+\alpha}}}\big)^{3/2}Q(\frac{\gamma}{\beta^{\frac{1}{1+\alpha}}}x)$$ into $\mathcal{E}_\beta(\varphi(x))$ one can easily check that
$$\limsup_{\beta_k\rightarrow 0} \frac{E(\beta_k,N_c)}{\beta^{\frac{1}{1+\alpha}}_k}\leq \frac{1}{2}\langle Q,\frac{m^2}{\sqrt{-\triangle} }Q\rangle^{\frac{\alpha}{1+\alpha}}\bigg(\int_{\mathbb{R}^3}(\frac{1}{|x|^\alpha}\ast|Q|^2)|Q|^2\,dx\bigg)^\frac{1}{1+\alpha}.$$
This means \eqref{energy-limit} holds. Combine \eqref{w-0-fom1} and \eqref{gamma}, the case (i) of Theorem \ref{theorm-asymp-beta} holds.

\section{Appendix}

\begin{lemma}[Lemma 2.4, \cite{F-J-Lenzmann2007}]\label{concen-compact}
Let $\{\psi_n\}$ be a bound sequence in $H^{1/2}(\mathbb{R}^3)$ such that $\mathcal{N}(\psi_n)=\int_{\mathbb{R}^3}|\psi_n|^2\,dx=N$ for all $n\geq 0.$ Then there exists a subsequence, $\{\psi_n\}$, satisfying one of the three following properties:\\
i) Compactness: There exists a sequence, $\{y_k\}$, in $\mathbb{R}^3$ such that, for every $\epsilon>0$, there exists $0<R<\infty$ with
\begin{align}
\int_{|x-y_k|<R}|\psi_{n_k}|^2\,dx\geq N-\epsilon.
\end{align}
ii) Vanishing:
\begin{align}
\limsup_{k\rightarrow \infty, y\in \mathbb{R}^3}\int_{|x-y|<R}|\psi_{n_k}|^2\,dx=0,\ \ \ \text{for all $R>0$}.
\end{align}
iii)Dichotomy: There exists $\lambda\in (0,N)$ such that, for every $\epsilon>0$, there exists two bounded sequences, $\{\psi^1_{n_k}\}$ and  $\{\psi^2_{n_k}\}$, in $H^{1/2}(\mathbb{R}^3)$ and $k_0\geq 0$ such that, for all $k\geq k_0$, the following properties holds:
\begin{align}\label{p-range}
\|\psi_{n_k}-(\psi^1_{n_k}-\psi^2_{n_k})\|_p\leq \delta_p(\epsilon),\ \ \text{for $2\leq p<3$}
\end{align}
with $\delta_p(\epsilon)\rightarrow 0$ as $\epsilon\rightarrow 0$, and
\begin{align}
\big|\int_{\mathbb{R}^3}|\psi^1_k|^2\,&dx-\lambda\big|\leq \epsilon,\ \ \ \big|\int_{\mathbb{R}^3}|\psi^2_k|^2\,dx-(N-\lambda)\big|\leq \epsilon,\\
&dist(supp\,\psi^1_k,supp\,\psi^2_k)\rightarrow \infty, \ \ \text{as $k\rightarrow \infty$.}
\end{align}
Moreover, we have that
\begin{align}\label{binding-est1}
\liminf_{k\rightarrow\infty}\big(\langle\psi_{n_k}, T\psi_{n_k}\rangle-\langle\psi^1_{n_k}, T\psi^1_{n_k}\rangle-\langle\psi^2_{n_k}, T\psi^2_{n_k}\rangle\big)\geq -C(\epsilon),
\end{align}
where $C(\epsilon)\rightarrow 0$ as $\epsilon\rightarrow 0$ and $T:=(\sqrt{-\triangle+m^2}-m)$ with $m\geq 0$.
\end{lemma}

\begin{lemma}\label{vanishing-alpha}
Let $\{\psi_n\}$ satisfy the assumptions of Lemma 2.4 in \cite{F-J-Lenzmann2007}. Furthermore, suppose that there exists a subsequence, still denoted by $\{\psi_n\}$, that satisfies part ii) of Lemma 2.4. Then for all $0<\theta<2$
\begin{align}
\lim_{n\rightarrow \infty}\int_{\mathbb{R}^3}(\frac{1}{|x|^\theta}\ast|\psi_n|^2)|\psi_n|^2\,dx=0;
\end{align}

\end{lemma}
\begin{proof}
The same arguments as Lemma A.1 of \cite{F-J-Lenzmann2007}, we can prove this lemma, we omit the detail here.
\end{proof}

\begin{lemma}
Suppose that $\epsilon>0$. Let $\{\psi_n\}$ satisfy the assumptions of Lemma \ref{concen-compact} and let $\{\psi_{n_k}\}$ be a subsequence that satisfies part iii) with sequences $\{\psi^1_{n_k}\}$ and $\{\psi^2_{n_k}\}$. Then for any $0<\theta<2$ ,
\begin{align}\label{binding-est2}
&\bigg|\int_{\mathbb{R}^3}(\frac{1}{|x|^\theta}\ast|\psi_{n_k}|^2)|\psi_{n_k}|^2\,dx-\int_{\mathbb{R}^3}(\frac{1}{|x|^\theta}\ast|\psi^1_{n_k}|^2)|\psi^1_{n_k}|^2\,dx-
\int_{\mathbb{R}^3}(\frac{1}{|x|^\theta}\ast|\psi^2_{n_k}|^2)|\psi^2_{n_k}|^2\,dx\bigg|\notag\\
&\leq r^\theta_1(k)+r^\theta_2(\epsilon),
\end{align}
for $k$ sufficiently large, where $r^\theta_1(k)\rightarrow 0$ as $k\rightarrow \infty$ and $r^\theta_2(\epsilon)\rightarrow 0$ as $\epsilon\rightarrow 0$.
\end{lemma}
\begin{proof}
Note that by Hardy-Littlewood-Sobolev inequality, interpolation inequality and Sobolev inequality,  we have
\begin{align*}
\bigg|\int_{\mathbb{R}^3}(\frac{1}{|x|^\theta}\ast|\psi|^2)|\psi|^2\,dx\leq C\|\psi\|^4_{\frac{12}{6-\theta}}\leq C_1 \|\psi\|^{2\theta}_{3} \|\psi\|^{2(2-\theta)}_{2}\leq C_2\langle\psi,\sqrt{-\triangle}\,\psi\rangle^{\theta}.
\end{align*}
Notice that $2<\frac{12}{6-\theta}<3$ including in the range of $p$ in \eqref{p-range}. Then the same arguments as Lemma A.2 of \cite{F-J-Lenzmann2007}, the lemma holds. We omit the detail arguments here.
\end{proof}
\begin{theorem}\label{lemma-m-0}
	If $m=0$ and $\beta>0$, there exists no minimizer for $E(\beta,N_c)$.
\end{theorem}
\begin{proof}
	 Let $Q^{\lambda}= \lambda^{3/2}Q(\lambda x)$ with $\lambda>0$, where $Q$ is the optimizer of  \eqref{maininequality}, by \eqref{maininequality} and \eqref{Pohoza-identity}
	\begin{align*}
	&0\leq E(\beta,N_c)\leq \mathcal{E}_\beta(Q^\lambda)\\
	&=\frac{1}{2}\langle Q^\lambda,(\sqrt{-\triangle})Q^\lambda\rangle-\frac{1}{4}\int_{\mathbb{R}^3}(\frac{1}{|x|}\ast|Q^\lambda|^2)|Q^\lambda|^2\,dx+
	 \frac{\beta}{4}\int_{\mathbb{R}^3}(\frac{1}{|x|^\alpha}\ast|Q^\lambda|^2)|Q^\lambda|^2\,dx\\
	&=\frac{\lambda}{2}\langle Q,(\sqrt{-\triangle})Q\rangle-\frac{\lambda}{4}\int_{\mathbb{R}^3}(\frac{1}{|x|}\ast|Q|^2)|Q|^2\,dx+
	 \frac{\beta\,\lambda^\alpha}{4}\int_{\mathbb{R}^3}(\frac{1}{|x|^\alpha}\ast|Q|^2)|Q|^2\,dx\\
	 &=\frac{\beta\,\lambda^\alpha}{4}\int_{\mathbb{R}^3}(\frac{1}{|x|^\alpha}\ast|Q|^2)|Q|^2\,dx\rightarrow 0, \ \ \text{as $\lambda\rightarrow 0$.}
	\end{align*}
	which means that $E(\beta,N_c)=0$ with $m=0$ and $\beta>0$, this implies that no minimizer exists. In fact, if there exists a minimizer $v\in H^{1/2}(\mathbb{R}^3)$ with $\|v\|^2_2=N_c$, then we have
	$$E(\beta,N_c)=\mathcal{E}_\beta(v)\geq \frac{\beta}{4}\int_{\mathbb{R}^3}(\frac{1}{|x|^\alpha}\ast|v|^2)|v|^2\,dx >0,$$
	which is a contradiction. Thus we complete this theorem.
\end{proof}

\end{document}